\documentclass[review]{elsarticle}
\usepackage{amsmath,amssymb,amsthm}
\usepackage{hyperref}

\journal{Journal of \LaTeX\ Templates}
\nonstopmode \numberwithin{equation}{section}
\newtheorem{theorem}{Theorem}[section]
\newtheorem{lemma}[theorem]{Lemma}
\newtheorem{proposition}[theorem]{Proposition}
\newtheorem{remark}[theorem]{Remark}
\newtheorem{corollary}[theorem]{Corollary}
\newtheorem{definition}[theorem]{Definition}

\newcommand{\norm}[1]{\left\|#1\right\|}
\newcommand{\abs}[1]{\left\lvert#1\right\rvert}

\begin{document}
	
	\begin{frontmatter}
		
		\title{Optimal Control and Approximate controllability of fractional  semilinear differential inclusion involving $\psi$- Hilfer fractional derivatives\tnoteref{mytitlenote}}
		\tnotetext[mytitlenote]{Fully documented templates are available in the elsarticle package on \href{http://www.ctan.org/tex-archive/macros/latex/contrib/elsarticle}{CTAN}.}
		
		\author{Bholanath Kumbhakar\fnref{myfootnote}}
		\address{Department of Mathematics, Indian Institute of Technology Roorkee}
		\fntext[myfootnote]{Email id: bkumbhakar@mt.iitr.ac.in}
		
		
		\author{Dwijendra Narain Pandey\corref{mycorrespondingauthor}}
		\cortext[mycorrespondingauthor]{Corresponding author}
		\address{Department of Mathematics, Indian Institute of Technology Roorkee}
		\ead{dwij@ma.iitr.ac.in}
		
		
		\begin{abstract}
		The current paper initially studies the optimal control of linear $\psi$-Hilfer fractional derivatives with state- dependent control constraints and optimal control for a particular type of cost functional. Then, we investigate the approximate controllability of the abstract fractional semilinear differential inclusion involving $\psi$-Hilfer fractional derivative in reflexive Banach spaces. It is known that the existence, uniqueness, optimal control, and approximate controllability of fractional differential equations or inclusions have been demonstrated for a similar type of fractional differential equations or inclusions with different fractional order derivative operators. Hence it has to research fractional differential equations with more general fractional operators which incorporate all the specific fractional derivative operators. This motivates us to consider the $\psi$-Hilfer fractional differential inclusion. We assume the compactness of the corresponding semigroup and the approximate controllability of the associated linear control system and define the control with the help of duality mapping. We observe that convexity is essential in determining the controllability property of semilinear differential inclusion. In the case of Hilbert spaces, there is no issue of convexity as the duality map becomes simply the identity map. In contrast to Hilbert spaces, if we consider reflexive Banach spaces, there is an issue of convexity due to the nonlinear nature of duality mapping.
		The novelty of this paper is that we overcome this convexity issue and establish our main result. Finally, we test our outcomes through an example.
		\end{abstract}
		
		
	\end{frontmatter}
	
	\section{Introduction}
This article discusses optimal control and the approximate controllability of fractional evolution inclusions involving $\psi$-Hilfer fractional derivatives.
We start with the optimal control problem given below: \\
\textbf{Problem I:} Find $(q^*(\cdot),u^*(\cdot))\in R_U(x_0)$ such that
\begin{equation}
	J(q^*(\cdot),u^*(\cdot))=\inf_{(q(\cdot),u(\cdot))\in R_U(x_0)} J(q(\cdot),u(\cdot)),
\end{equation}
where $R_U(x_0)$ is the solution set of the feedback control system involving $\psi$-Hilfer fractional derivative given by
\begin{equation}
	\begin{cases}
		^{H}D^{\alpha,\beta;\psi}_{a+} q(t)=Aq(t)+Bu(t), t\in [a,b]\\
		(I_{a+}^{1-\gamma;\psi}q)(a)=x_0,
	\end{cases}
\end{equation}
where $u(t)\in U(t,q(t))$ and $J:C^{1-\gamma;\psi}([a,b],X)\times L^2([a,b],Y)\to \mathbb{R}$ is defined by
\begin{equation}
	J(q(\cdot), u(\cdot))=\int_T h(t,q(t),u(t))dt, \forall (q(\cdot),u(\cdot))\in C^{1-\gamma;\psi}([a,b],X)\times L^2([a,b],Y).
\end{equation}
In the above, $A$ is a closed linear operator generating a strongly continuous semigroup $\{T(t)\}_{t\ge 0}$ of bounded linear operators defined on a Banach space $X$ and $^{H}D^{\alpha,\beta;\psi}_{a+}$ is the left $\psi$-Hilfer fractional derivative of order $\alpha$ and type $\beta$ with $\frac{1}{2}<\alpha\le1$, $0\le \beta\le 1$. Also, $B: Y\to X$ is a bounded linear map that describes the control action, $Y$ being a separable Hilbert space, $U:[a,b]\times X\multimap Y$ is a multivalued map that describes the control constraints. The space $C^{1-\gamma;\psi}([a,b],X)$ is defined as
\begin{equation}
	C^{1-\gamma;\psi}([a,b],X)=\{y\in C((a,b],X): (\Psi(t,a))^{1-\gamma}y(t)\in C([a,b],X)\},
\end{equation}
$ \gamma=\alpha+\beta(1-\alpha).$
We study existence as well as uniqueness of Problem I.

Then we consider a particular case of Problem I, namely\\
\textbf{Problem II:} 
We obtain the existence of optimal control by minimizing the cost functional $J: C^{1-\gamma;\psi}([a,b], X)\times L^2([a,b], Y)\to \mathbb{R}$ given by
\begin{equation}
	J(q(\cdot),u(\cdot))=\norm{q(b)-x_b}_X^2+\lambda \int_{a}^{b}\norm{u(t)}_Y^2dt,
\end{equation}
where $x_b\in X, \lambda>0$ and $q(\cdot)$ is the unique mild solution of the fractional linear control problem involving $\psi$-Hilfer fractional derivative given by
\begin{equation}\label{A1}
	\begin{cases}
		^{H}D^{\alpha,\beta;\psi}_{a+} q(t)=Aq(t)+Bu(t), t\in [a,b]\\
		(I_{a+}^{1-\gamma;\psi}q)(a)=x_0.
	\end{cases}
\end{equation}
In this problem, we also derive the explicit expression of the optimal control.
\begin{remark}
	Note that the solution of Problem II also gives us the approximate controllability for linear control system \eqref{A1}. Indeed, when minimizing the functions given in
	Problem II, we are minimizing the balance between the two terms appearing there. The first term allows the
	state $q$ at time $b$ to get closer to the target state $x_b$ (approximate controllability of linear system),
	while the second one penalizes the use of costly control.
\end{remark}

Finally, we derive the approximate controllability of fractional differential evolution inclusion involving $\psi$-Hilfer fractional derivative given by\\
\textbf{Problem III:}
\begin{equation}\label{A2}
	\begin{cases}
		^{H}D^{\alpha,\beta;\psi}_{a+} q(t)\in Aq(t)+F(t,q(t))+Bu(t), t\in [a,b]\\
		(I_{a+}^{1-\gamma;\psi}q)(a)=x_0.
	\end{cases}
\end{equation} 
Here $F:[a,b]\times X\multimap X$ is given a multivalued map. We use the optimal control we derive in Problem II to study the approximate controllability for Problem III.

It is worth noting that most existing research works focused on integer order dynamics. However, in the real world, it has been shown that many natural phenomena cannot be effectively interpreted by the integer order dynamics, such as chemotaxis behavior and food searching of germs \cite{wu2019reach,cohen2001biofluiddynamics,kozlovsky1999lubricating}. Nevertheless, fractional order dynamics possess excellent memory and hereditary properties, resulting in superior performance and stronger robustness than standard integer order dynamic systems \cite{ozalp2011fractional},\cite{soczkiewicz2002application}. It has been proved that some cases, such as the macromolecule fluids, lateral inhibition of biological vision systems, and automobiles running on the road's surface containing viscoelastic materials, can be more accurately described by fractional-order dynamic systems \cite{bagley1983fractional,lu2013stability,ren2011distributed}. Researchers pointed out that many physical systems are unsuitable to be characterized by integer order dynamics \cite{zhang2016stability}, such as high-speed aircraft traveling on rainy days or snowy days and vehicles moving on top of sand or muddy road \cite{bagley1986fractional,anastassiou2012fractional}.

Most of the results available only focus on linear models or nonlinear models with continuous inherent
dynamics. However,
in real life, some phenomena may be well described by discontinuous dynamics \cite{wang2020global}. For example, thermostats implement on-off controllers to regulate room temperature \cite{bennett1993history}. The controller is a discontinuous
function of the room temperature. In nonsmooth mechanics, the motion of rigid bodies is subject to velocity jumps and force discontinuities as a result of friction and impact\cite{brogliato1999nonsmooth,pfeiffer1996multibody}. In robotic manipulation of objects utilizing mechanical contact, discontinuities occur naturally from interaction with the environment\cite{pereira2004decentralized}. In the physical field, the characteristics of an ideal diode possessing a very high slope in the conducting region can be more precisely modeled by a differential equation with a discontinuous right-hand side. For discontinuous dynamic systems, a continuously differentiable solution is not guaranteed. In this case, the theory of differential inclusion comes into the picture. Discontinuities also arise in models governed by fractional differential equations. The article \cite{forti2003global} considered the global leader-following consensus of fractional order multi-agent systems, where the inherent dynamics are modeled to be discontinuous. In \cite{ding2020finite}, the authors studied drive-response synchronization in fractional order memristine neural networks with switching jumps (FMNNs) mismatch. Since the right-hand sides of the equations of FMNNs are discontinuous, FMNNs have no solution in the ordinary sense, and theories of fractional order Filippov differential inclusions are used to treat FMNNs.

Many fractional derivatives have been developed due to Riemann, Liouville, Riesz, Grunwald, Letnikov, Marchaud, Weyl, Caputo, Hadamard, and other famous researchers \cite{podlubny1999fractional}. This led to several fractional differential equations. Each and every fractional derivative is essential in its own field. The most used concept is the fractional derivative in the sense of Riemann-Liouville. In \cite{heymans2006physical}, the authors applied Riemann-Liouville fractional derivatives in the context of viscoelastic models. On a series of examples
from the field of viscoelasticity they
demonstrate that it is possible to
attribute physical meaning to initial
conditions expressed in terms of
Riemann–Liouville fractional derivatives, and that it is possible to obtain initial values for such initial
conditions by appropriate measurements or observations. In the Caputo sense, the concept of a fractional derivative allows us to formulate the Cauchy problem similar to what occurs in the abstract Cauchy problem with a derivative of integer order. The $\psi$-Caputo fractional derivative has a wide range of applications in various fields of science and engineering, including fluid dynamics, electromagnetism, finance, and control systems. Some specific applications of the $\psi$ Caputo fractional derivative include: modeling of viscoelastic materials \cite{han2022applications}, fractional-order control systems, population models \cite{awadalla2021psi}, modeling drug concentration in blood \cite{awadalla2022modeling}, fractional thermostat model \cite{aydi2020positive} and so on.
The authors in \cite{joshi2021chaos} discussed the chaos of calcium diffusion in Parkinson's infectious disease model and treatment mechanism via Hilfer fractional derivative. The article \cite{singh2020analysis} deals with certain new and exciting features of the fractional blood alcohol model associated with the powerful Hilfer fractional operator. 

Therefore, we need to unify all these fractional derivatives into one most general system. Sousa and Olivera \cite{sousa2018psi} presented a Hilfer version of fractional derivative, viz, $\psi$-Hilfer fractional derivative, which incorporates the Hilfer fractional derivative \cite{hilfer2000applications} as well as includes a broad class of well known fractional derivatives including most widely used Caputo and Riemann-Liouville fractional derivative. The list of all possible fractional derivatives, the cases of $\psi$-Hilfer fractional derivatives, has been provided in \cite{sousa2018psi}. This new type of derivative is rich in applications. The authors in \cite{thaiprayoon2021qualitative} presented an application of $\psi$-Hilfer fractional derivative in a fractional thermostat model.

As is commonly known, mathematical control theory has several fundamental properties, one of which is controllability. Controllability roughly translates to using the control function incorporated in the system to steer the dynamical system's state to an appropriate state. Controllability is essential in modern control theory and engineering because it is connected to pole assignment, structural decomposition, and optimal quadratic control. A comprehensive controllability theory and applications review may be found in \cite{curtain1978infinite}.

Several controlling concepts should be distinguished, for instance, exact controllability, approximate controllability, null controllability, and optimal control. 
Among exact, null, and approximate controllability, approximate controllability is more beneficial in a real-life situation. The approximate controllability of an abstract semilinear system is more important in population dynamics than exact controllability. The optimal control of a fractional distributed parameter system is an optimal control for which system dynamics are defined with fractional differential equations. There is a vast literature towards existence, uniqueness and controllability for fractional differential equations/inclusions. For instance, we refer the readers to the following papers: \cite{yang2022optimal},\cite{lian2018time},\cite{yang2016approximate} for Riemann-Lioville fractional differential equations , \cite{matychyn2018time},\cite{wang2011class},\cite{sakthivel2013approximate} for Caputo, for Hilfer \cite{yang2016approximate}, and for other type of derivatives we cite \cite{wang2011analysis},\cite{kien2022optimal},\cite{qin2014approximate},\cite{yang2022optimal} and references cited therein.

One can see that existence, uniqueness, and approximate controllability for fractional differential equations have been demonstrated for a similar type of fractional differential equations (FDEs) with different fractional order derivative operators. Hence it has to research FDEs with more general fractional operators which incorporate all the specific fractional derivative operators. Therefore it is imperative to analyze the FDE with a broad class of fractional derivative operators that comprise various definitions of well-known fractional derivatives. To this aim, recently, the authors in \cite{bouacida2023controllability} studied the approximate controllability for $\psi$-Hilfer backward fractional differential equations. Also, the authors in \cite{guechi2021analysis} analyzed optimal control problem for $\psi$-Hilfer fractional differential equations.  To our best knowledge, there is no literature for studying the approximate controllability and optimal control problems driven by $\psi$-Hilfer fractional evolution inclusion. Therefore, to fill this gap, we initially consider a nonlinear optimal control problem governed by $\psi$-Hilfer linear control problem with state dependent control constraints (Problem I) and establish the existence theorem for this linear control problem. Next, we consider $\psi$-Hilfer fractional evolution inclusion (Problem III) and study the approximate controllability result.

Before going into the crux of the main result, we briefly mention the working procedure for approximate controllability.  Throughout the manuscript we use the convention that $\psi(t)-\psi(s)=\Psi(t,s)$.

Let $x_1\in X$ be the desired target.
Define $S_F:C^{1-\gamma;\psi}([a,b],X)\multimap L^1([a,b],X)$
\begin{equation}\label{1.8}
	S_F(q)=\{f\in L^1([a,b],X): f(t)\in F(t,q(t))~\text{a.a.}~t\in [a,b]\}.
\end{equation}
The notion of a mild solution of \eqref{A2} is given by the following definition:
\begin{definition}
	We say that $q(\cdot)\in C^{1-\gamma;\psi}([a,b],X)$ is a mild solution of the given fractional control problem \eqref{A2} if it satisfies the following integral equation
	\begin{equation}
		q(t)= S_{\alpha, \beta}(\Psi(t,a))x_0 + \int_{a}^{t}  K_{\alpha}(\Psi(t,s)) [f(s)+Bu(s)] \psi'(s)ds,
	\end{equation}
	where $f\in S_F(q)$ and the operators $S_{\alpha, \beta}(t) : X\to X ~ and~  K_{\alpha}(t): X \to X $  are defined by 
	\begin{equation}
		S_{\alpha, \beta}(t)= I^{\gamma-\alpha;\psi}_{a+}K_{\alpha}(t),	K_{\alpha}(t)= t^{\alpha-1} P_{\alpha}(t)~~
		\text{where}~~
		P_{\alpha}(t)= \int_{0}^{\infty} \alpha \theta  M_{\alpha}(\theta)  T(t^{\alpha}\theta) d\theta.
	\end{equation}
\end{definition}
For each $\epsilon>0$, we define a multivalued map $$\Gamma_{\epsilon}:C^{1-\gamma;\psi}([a,b],X)\multimap C^{1-\gamma;\psi}([a,b],X)$$ where for each $q\in C^{1-\gamma;\psi}([a,b],X)$, $\Gamma_{\epsilon}(q)$ consists functions $y\in C^{1-\gamma;\psi}([a,b],X)$ which satisfies
\begin{equation}\label{A3}
	y(t)=S_{\alpha,\beta}(\Psi(t,a))x_0+\int_{a}^{t}\psi^{\prime}(s)(\Psi(t,s))^{\alpha-1}P_{\alpha}(\Psi(t,s))[f(s)+Bu(s)]ds, 
\end{equation}
$t\in [a,b],$ where $f\in S_F(q)$ and $u\in L^2([a,b],Y)$ is given by
\begin{equation}\label{A4}
	u(t)=\psi^{\prime}(t)(\psi(b)-\psi(t))^{\alpha-1}B^*P^*_{\alpha}(\psi(b)-\psi(t))J(\epsilon I +R(b)J)^{-1}N(f), t\in [a,b].
\end{equation}
In the above $R(b):X^*\to X$ is given by
\begin{equation}
	R(b)=\int_{a}^{b}\left\{\psi^{\prime}(s)(\Psi(b,s))^{\alpha-1}\right\}^2P_{\alpha}(\Psi(b,s))BB^*P_{\alpha}^*(\Psi(b,s))ds,
\end{equation}
and
\begin{equation}\label{A5}
	N(f)=x_1-S_{\alpha,\beta}(\Psi(b,a))x_0-\int_{a}^{b}\psi^{\prime}(s)K_{\alpha}(\Psi(b,s))f(s)ds.
\end{equation}
The map $J: X\multimap X^*$ is called duality mapping, which is defined as
\begin{equation}
	J(x)=\{x^*\in X^*: \langle x, x^*\rangle=\norm{x}_X^2=\norm{x^*}^2_{X^*}\}, \forall x\in X.
\end{equation} 
Since the space $X$ is a reflexive Banach space, then $X$ can be renormed so that $X$ and $X^*$ become strictly convex \cite{asplund1967averaged}. As a result, the mapping $J$ becomes single valued as well as demicontinuous, that is, 
\begin{equation}
	x_k\to x ~\text{in}~ X ~\text{implies}~ J(x_k)\rightharpoonup J(x) ~\text{in}~ X^*, ~\text{as}~k\to \infty.
\end{equation}
The fixed points of the map $\Gamma_{\epsilon}$ are the mild solutions of \eqref{A2}.

To prove the fixed point of the map $\Gamma_{\epsilon}$ the following difficulties may arise:
\begin{itemize}
	\item[(1)] Note that for any $q\in C^{1-\gamma;\psi}([a,b], X)$, $\Gamma_{\epsilon}(q)$ fails to be convex due to the nonlinear nature of the duality map. It is well known that convexity plays a crucial role in most multivalued fixed point theorems. Therefore, one may be unable to apply the multivalued fixed point theorem to derive the fixed point of the map. To overcome this problem, we use Schauder fixed point theorem to derive the result.
	\item[(2)] Upper semicontinuity of the multivalued map $F$ is insufficient to guarantee fixed points' existence. We need lower semicontinuity also.
	\item[(3)] Further, unlike control problems governed by ordinary differential equations of integer order, the control problems generated by fractional differential equations are quite complicated. We meet some difficulties in estimating the bounds of state variables. The main reason is the appearances of the $L^1$-function $[0,t]\ni s \mapsto\frac{1}{(t-s)^{1-\alpha}}$. To overcome these difficulties, we shall use a generalized Gronwall inequality. 
\end{itemize}
We overcome these issues and successfully applying Schauder's fixed point theorem we prove the existence of fixed points, say $q_{\epsilon}\in C^{1-\gamma;\psi}([a,b],X)$ for the maps $\Gamma_{\epsilon}$, for each $\epsilon>0$. Clearly, each $q_{\epsilon}, \epsilon>0$ are the mild solutions of Problem \eqref{A2}. We conclude our result by proving $q_{\epsilon}(b)\to x_1$.

The article is organised as follows. Section 2 presents some prerequisite materials which we use to deduce the main results. Section 3 deals with optimal control of $\psi$-Hilfer fractional linear control problem. Section 4 stands for approximate controllability of $\psi$-Hilfer fractional evolution inclusion in reflexive Banach spaces. Finally, Section 5 presents an application of Main results of Section 4 and closes this article.

\section{Preliminaries}
We first introduce some basic definitions, notations, preliminaries, results, etc., and then we give some helpful assumptions and lemmas throughout this article.
Here $X$  always denotes a Banach space induced by the norm $\norm{\cdot}$ and  $C([a,b],X)$ denotes the Banach space of continuous functions from $[a,b]$ to $X$  with the norm $\norm{q(\cdot)}_{C([a,b],X)}=\sup_{t \in [a,b]}\norm{q(t)}_X$.\\
Let us introduce the space provided in \cite{sousa2018psi}
\begin{equation}
	C^{1-\gamma;\psi}([a,b],X)=\{q(\cdot)\in C((a,b],X): (\Psi(t,a))^{1-\gamma}q(t)\in C([a,b],X) \}.
\end{equation}
Then $C^{1-\gamma;\psi}([a,b],X)$ is a Banach space under the norm
\begin{equation}
	\norm{q(\cdot)}_{C^{1-\gamma;\psi}([a,b],X)}=\sup_{t \in [a,b]}\norm{(\Psi(t,a))^{1-\gamma}q(t)}.
\end{equation}
Throughout this article, all the integrations are taken in the Bochner sense.

We give the definition of $\psi$-Hilfer fractional derivative.
\begin{definition}($\psi$-Hilfer Fractional Derivative)\cite{sousa2018psi}
	Let $n-1<\alpha<n,n \in \mathbb{N}, -\infty \le a<b\le \infty$. Also, let $F,\psi \in C^n([a, b],X) $ such that $\psi$ is increasing and  $\psi'(t) \neq 0$ for all $t \in [a,b]$. The $\psi$-Hilfer fractional derivative $	^{H}D_{\sigma+}^{\alpha, \beta ; \psi}$ of order $\alpha$ and type $\beta$ with $0\le \beta \le 1$ of $F$ is defined by
	\begin{equation*}
		^{H}D_{a+}^{\alpha, \beta ; \psi}  F(t)=  I^{\beta(n-\alpha);\psi}_{a+}\left( \frac{1}{\psi'(t)}\frac{d}{dt}\right)^n I^{(1-\beta)(n-\alpha);\psi}_{a+}F(t).
	\end{equation*}
\end{definition}
It is well known that $\psi$-Hilfer fractional derivative is the most general fractional derivative.

\begin{theorem}(Schauder’s fixed point theorem)
	Let $X$ be a Banach space and $D \subset $X, a convex, closed, and bounded set. If $T :  D\to D$ is a continuous operator such that
	$T(D)$ is relatively compact in $X$, then $T$ has at least one fixed point in $D$.
\end{theorem} 
Let $B(Y)$ be the $\sigma$- algebra of Borel sets in $Y$ and let $\Sigma\times B(Y)$ be the $\sigma$-algebra of sets in $\Omega\times Y$ generated by sets $A\times B$, where $A\in \Sigma,$ and $B\in B(Y)$.
\begin{definition}
	A multivalued mapping $\Gamma: \Omega\times Y\multimap X$ is called $\Sigma\times B(Y)$ measurable if
	\begin{equation*}
		\Gamma^{-1}(V)=\{(w,y)\in \Omega\times Y: \Gamma (w,y)\cap V\neq \phi\}\in \Sigma\times B(Y),
	\end{equation*}
	for any closed set $V\subset X$.\\
\end{definition} 
For more details on the multivalued maps, we refer to the reader the books \cite{papageorgiou1997handbook, aubin2012differential} and the references cited therein.\\
Next we recall that a subset $\Sigma\subset L^1([0,\nu];X)$ is called uniformly integrable if for every $\epsilon>0$ there is a $\delta(\epsilon)>0$ such that
\begin{equation*}
	\int_{E}\norm{f(s)}ds\le \epsilon,
\end{equation*}
for every measurable subset $E\subset [0,\nu]$ whose Lebesgue measure is less than or equal to $\delta(\epsilon)$, and uniformly with respect to $f\in \Sigma$.\\
\begin{remark}
	\begin{itemize}
		\item[(1)]  A simple argument involving H$\ddot{o}$lder's inequality shows that each bounded subset $\Sigma\subset L^p([0,\nu];X)$ with $1<p\le \infty$ is uniformly integrable.
		\item[(2)]  If $\Sigma$ is a subset of $L^1([0,\nu];X)$, for which there exists a $g\in L^1([0,\nu];\mathbb{R}^+)$ such that $\norm{f(t)}\le g(t)$, for every $f\in \Sigma$ and a.a. $t\in [0,\nu]$, then $\Sigma$ is uniformly integrable.
	\end{itemize}
\end{remark}
We need some weak compactness criterion in $L^1$ space. The best-known result in this direction is the celebrated Dunford-Pettis theorem.
\begin{theorem}[Dunford Pettis theorem]\label{DF}\cite{MR1132408}
	Let $X$ be a Banach space and $\Sigma\subset L^1([0,\nu],X)$ be uniformly integrable. Suppose there exists $C:[0,\nu]\multimap X$ such that $\sigma(t)\in C(t)$ for every $\sigma\in \Sigma$, and $C(t)$ is relatively weakly compact for all $t\in [0,\nu]$, then $\Sigma$ is relatively weakly compact $L^1([0,\nu],X)$.
\end{theorem}

\section{Hypotheses}
Throughout the article, we assume that $\frac{1}{2}<\alpha\le 1$ and consider the following hypotheses.\\
\begin{itemize}
	\item[(T)]  The operator $A$ generates a $C_0$ semigroup $T(t)$ which is compact for $t > 0$.
	\item[(C)] The linear fractional control system \eqref{A1} is approximately controllable in $[a,b]$.
	\item[(F1)] The multimap $(t,x)\mapsto F(t,x)$ is graph measurable.
	\item[(F2)] The multimap $F(t,x)$ has closed and convex values.
	\item[(F3)]  For each $t\in [a,b]$, the multimap $F(t,\cdot):X \multimap X$  is lower semicontinuous.
	\item[(F4)]  There exists a function $m \in L^{\frac{1}{r}}([a,b],\mathbb{R}^+), r\in (0,\alpha)$ such that
	\begin{equation*}
		I^{\alpha;\psi}_{a+} m \in C((a,b], \mathbb{R}^+ ), \lim_{t\to a+}(\Psi(t,a))^{1-\gamma}  I^{\alpha;\psi}_{a+} m(t)=0,
	\end{equation*} 
	with
	\begin{equation*}
		\norm{F(t,x)} \le m(t), ~\forall~ x \in X~ and~  t \in [a,b].
	\end{equation*}
\end{itemize}
\begin{remark}
	A graph measurable
	multifunction $F : [a, b]\times X\multimap X$ 
	has the property that if $q: [a,b]\to X$ is measurable, then $t\to F(t,q(t))$ is graph measurable. So, by Aumann's selection theorem \cite{wagner1977survey}, we can find a measurable function $f:[a, b]\to X$ such that $f(t)\in F(t,q(t))$ a.e. $t\in [a, b]$. Therefore, the measurable selection map $S_F$ defined by \eqref{1.8} is well defined.
\end{remark}
We assume that $U:[a,b]\times X\multimap Y$ satisfies
\begin{itemize}
	\item[(U1)] $(t,x)\mapsto U(t,x)$ is measuarble.
	\item[(U2)] The set $U(t,x)$ is convex in $Y$ for all $(t,x)\in [a,b]\times X$.
	\item[(U3)] $\text{For all}~ x\in X, ~\text{a.e.}~ t\in [a,b]$,
	\begin{equation*}
		\norm{U(t,x)}_{Y}\le a_U(t) +c_U(\Psi(t,a))^{1-\gamma}\norm{x}_{X}~~\text{with}~a_U\in L^2([a,b],\mathbb{R}^+), c_U\ge 0.
	\end{equation*}
	\item[(U4)] 	For all $x,y\in X$, a.e. $t\in [a,b]$,
	\begin{equation*}
		h(U(t,x), U(t,y))\le k_4(t)(\Psi(t,a))^{1-\gamma}\norm{x-y}_{X},
	\end{equation*}
	$k_4(\cdot)\in L^1([a,b],\mathbb{R}^+)$.
\end{itemize}
Also, we suppose that $h: [a,b]\times X\times Y\to \mathbb{R}$ satisfies the following conditions:
\begin{itemize}
	\item[(H1)] for all $(x,u)\in X\times Y$, the map $t\to h(t,x,u)$ is measurable;
	\item[(H2)] there exists $k_1, k_2\in L^1([a,b],\mathbb{R}^+)$ and $c_h\ge 0$ such that
	\begin{equation}
		\abs{h(t,x,u)}\le k_1(t)+k_2(t)\norm{x}_X+c_h\norm{u}_Y,
	\end{equation} 
	for all $(x,u)\in X\times Y$ and a.e. $t\in [a,b]$;
	\item[(H3)] for a.e. $t\in [a,b]$ and $u\in Y$, the function $x\mapsto h(t,x,u)$ is lower semicontinuous;
	\item[(H4)] for a.e. $t\in [a,b]$ and $x\in X$, the function $u\mapsto h(t,x,u)$ is lower semicontinuous and convex;
	\item[(H5)] for each bounded set $D\subset Y$, there exists a function $a_D\in L^2([a,b],\mathbb{R}^+)$ such that
	\begin{equation}
		\abs{h(t,x,u(t))-h(t,y,u(t))}\le a_D(t)\norm{x-y}_X,
	\end{equation}
	for all $x,y\in X$, $u\in D$ and a.e. $t\in [a,b]$. 
\end{itemize}
\begin{remark}
	One can choose $U:[a,b]\times L^2([0,1])\multimap L^2([0,1])$ given by
	\begin{equation}
		U(t,x)=\{(\Psi(t,a))^{1-\gamma}y: \norm{y-g(x)}\le \rho(x)\}, t\in [a,b], x\in X=L^2([0,1]),
	\end{equation}
	where $g: X\to X$ and $\rho:X\to \mathbb{R}$ satisfies $g(0)=0, \rho(0)=0$ and 
	\begin{equation}
		\norm{g(x)-g(y)}\le K\norm{x-y}, x,y\in X,
	\end{equation}
	and 
	\begin{equation}
		\abs{\rho(x)-\rho(y)}\le K\norm{x-y}, x,y\in X.
	\end{equation}
	We show $U$ satisfies hypotheses (U1)-(U4). Clearly, Hypothesis (U1) and (U2) are satisfied. For (U3), we consider $u\in U(t,x)$. Then
	\begin{equation}
		\norm{u}=\norm{(\Psi(t,a))^{1-\gamma}y}=(\Psi(t,a))^{1-\gamma}\norm{y}.
	\end{equation}
	Using the condition $g(0)=0$, we obtain
	\begin{align*}
		\norm{u}\le (\Psi(t,a))^{1-\gamma}[\norm{y-g(x)}+\norm{g(x)}\le 2(\Psi(t,a))^{1-\gamma}\norm{x},
	\end{align*}
	which implies that Hypothesis (U3) is satisfied.
	
	We now verify Hypothesis (U4). Take $x,y\in X$ and $u\in U(t,x)$. Then $u=(\Psi(t,a))^{1-\gamma}u_0$, where $\norm{u_0-g(x)}\le \rho(x)$. Let
	\begin{equation}
		v_0=g(y)+\frac{\rho(y)}{\rho(x)}(u_0-g(x)).
	\end{equation}
	Then $v=(\Psi(t,a))^{1-\gamma}v_0\in U(t,y)$ and
	\begin{align*}
		\norm{u_0-v_0}=&\norm{u_0-g(y)-\frac{\rho(y)}{\rho(x)}(u_0-g(x))}\\
		=&\norm{g(x)+u_0-g(x)-g(y)-\frac{\rho(y)}{\rho(x)}(u_0-g(x))}\\
		=&\norm{g(x)-g(y)+\left[1-\frac{\rho(y)}{\rho(x)}(u_0-g(x))\right]}
		\le 2L\norm{x-y}.
	\end{align*}
	Hence
	\begin{equation}
		\norm{u-v}\le2L (\Psi(t,a))^{1-\gamma}\norm{x-y}, x,y\in X.
	\end{equation}
	Repeating the above process, we obtain
	\begin{equation}
		d_H(U(t,x),U(t,y))\le2L(\Psi(t,a))^{1-\gamma}\norm{x-y}, x,y\in X. 
	\end{equation}
\end{remark}
\section{Optimal Control}
In this section, we study the linear fractional control problem
\begin{equation}\label{4.1}
	\begin{cases}
		^{H}D^{\alpha,\beta;\psi}_{a+} q(t)=Aq(t)+Bu(t), t\in [a,b]\\
		(I_{a+}^{1-\gamma;\psi}q)(a)=x_0,
	\end{cases}
\end{equation}
with the control constraint
\begin{equation}\label{4.2}
	u(t)\in U(t,q(t))~\text{for a.a.}~t\in [a,b].
\end{equation}
Here $A$ is a closed linear operator generating a strongly continuous semigroup $\{T(t)\}$ of bounded linear operators defined on a Banach space $X$. $^{H}D^{\alpha,\beta;\psi}_{a+}$ is the left $\psi$-Hilfer fractional derivative, $U:[a,b]\times X\multimap Y$,  $B: Y\to X$ is a bounded linear map that describes the control action, $Y$ being a separable Hilbert space. 

The notion of mild solution of \eqref{4.1}-\eqref{4.2} is given by the following definition:
\begin{definition}
	We say that $(q(\cdot), u(\cdot))\in C^{1-\gamma;\psi}([a,b],X)\times L^2([a,b],Y)$ is a mild solution of the given fractional control problem \eqref{4.1}-\eqref{4.2} if it satisfies the following integral equation
	\begin{equation}\label{4.3}
		q(t)= S_{\alpha, \beta}(\Psi(t,a))x_0 + \int_{a}^{t}  K_{\alpha}(\Psi(t,s)) Bu(s) \psi'(s)ds,
	\end{equation}
	with 
	\begin{equation}
		u(t)\in U(t,q(t))~\text{a.a.}~t\in [a,b],
	\end{equation}
	where the operators $S_{\alpha, \beta}(t) : X\to X ~ and~  K_{\alpha}(t): X \to X $  are defined by 
	\begin{equation}
		S_{\alpha, \beta}(t)= I^{\gamma-\alpha;\psi}_{a+}K_{\alpha}(t),
		K_{\alpha}(t)= t^{\alpha-1} P_{\alpha}(t)~~
		\text{where}~~
		P_{\alpha}(t)= \int_{0}^{\infty} \alpha \theta  M_{\alpha}(\theta)  T(t^{\alpha}\theta) d\theta.
	\end{equation}
\end{definition}
The existence of mild solution of \eqref{4.1}-\eqref{4.2} is given in \cite{bouacida2023controllability}.
Let $R_U(x_0)$ be the set of all solutions of \eqref{4.1}-\eqref{4.2}.

In this section, we are interested in the following optimal control problem\\
\textbf{Problem 1} Find $(q^*(\cdot),u^*(\cdot))\in R_U(x_0)$ such that
\begin{equation}
	J(q^*(\cdot),u^*(\cdot))=\inf_{(q(\cdot),u(\cdot))\in R_U(x_0)} J(q(\cdot),u(\cdot)),
\end{equation}
where $J: C^{1-\gamma;\psi}([a,b],X)\times L^2([a,b],Y)\to \mathbb{R}$ is defined by
\begin{equation}
	J(q(\cdot), u(\cdot))=\int_{a}^{b}h(t,q(t),u(t))dt, \forall (q(\cdot),u(\cdot))\in C^{1-\gamma;\psi}([a,b],X)\times L^2([a,b],Y).
\end{equation}
Let us provide some important properties of the operators $S_{\alpha,\beta}(t)$ and $K_{\alpha}(t)$.
\begin{proposition}\label{Prop4.3}  \cite{gu2015existence}
	Under hypothesis (T), for any $t>0$, the operators  $S_{\alpha, \beta}(t)$ and $K_{\alpha}(t)$ are  linear operators and for any $x \in X$ 
	\begin{equation}
		\norm{ S_{\alpha, \beta}(t)x} \le  t^{\gamma-1} \frac{M}{\Gamma(\gamma)} \norm{x},
	\end{equation}
	and
	\begin{equation}
		\norm{K_{\alpha}(t)x}  \le t^{\alpha-1}\frac{M}{\Gamma(\alpha)} \norm{x}.
	\end{equation}
\end{proposition}	
\begin{proposition}\label{Prop4.4} \cite{gu2015existence}
	Under hypothesis (T), for any $t>0$, the operators  $K_{\alpha}(t)$  and  $S_{\alpha, \beta}(t)$ are strongly continuous, that means for any $x \in X$ and $0 <t_1<t_2\le b$ we have
	\begin{equation*}
		\norm{K_{\alpha}(t_1)x-K_{\alpha}(t_2)x} \to 0 ~as ~t_2 \to t_1,
	\end{equation*}
	and
	\begin{equation*}
		\norm{S_{\alpha,\beta}(t_1)x-S_{\alpha, \beta}(t_2)x} \to 0 ~as~ t_2 \to t_1.
	\end{equation*}
	
\end{proposition}	
\begin{proposition}\cite{gu2015existence}
	Under hypothesis (T), the operator $P_{\alpha}(t)$ is compact for $t>0$.
\end{proposition}
We now give a technical lemma.
\begin{lemma}\label{tech}
	The function $s\mapsto \psi^{\prime}(s)(\Psi(t,s)^{\alpha-1}$ belongs to $L^2([a,t],\mathbb{R})$, for all $t\in (a,b]$ provided $\frac{1}{2}<\alpha\le 1$.
\end{lemma}
\begin{lemma}\label{Lem 5.8}
	Suppose that the operator $P_{\alpha}(t)$ is compact for $t>0$. Let the operator $\Lambda: L^2([a,b],X)\to C^{1-\gamma;\psi}([a,b],X)$ be defined as
	\begin{equation}
		(\Lambda f)(t)=\int_{a}^{t}\psi^{\prime}(s)K_{\alpha}(\Psi(t,s))f(s)ds, t\in (a,b].
	\end{equation}
	Then the operator $\Lambda$ is compact.
\end{lemma}
\begin{proof}
	Let us first consider the ball
	\begin{equation}
		B_R=\{f\in L^2([a,b],X): \norm{f}_{L^2([a,b],X)}\le R\}.
	\end{equation}
	We will show that the set $\{\Lambda f: f\in B_R\}$ is relatively compact in $C^{1-\gamma;\psi}([a,b],X)$, that is, the set $\{(\psi(\cdot)-\psi(a))^{1-\gamma}\Lambda f(\cdot): f\in B_R\}\subset C([a,b],X)$ is relatively compact. First we show the set $\{(\psi(\cdot)-\psi(a))^{1-\gamma}\Lambda f(\cdot): f\in B_R\}\subset C([a,b],X)$ is equicontinuous.

	We define $\Xi:L^2([a,b],X)\to C([a,b],X)$
	\begin{equation}
		(\Xi f)(t)=(\Psi(t,a))^{1-\gamma}\int_{a}^{t}\psi^{\prime}(s)K_{\alpha}(\Psi(t,s))f(s)ds, t\in [a,b].
	\end{equation}
	We show the set $\Xi(B_R)\subset C([a,b],X)$ is equicontinuous.
	
	For $t_1=a$ and $a<t_2\le b$ we have
	\begin{align*}
		\norm{((\Xi f)(a)-(\Xi f)(t_2)}=&\norm{ (\Psi(t_2,a))^{1-\gamma}\int_{a}^{t_2}  K_{\alpha}(\Psi(t_2,s))  f(s)\psi'(s) ds}\\
		\le & (\Psi(t_2,a))^{1-\gamma} \int_{a}^{t_2} \norm{K_{\alpha}(\Psi(t_2,s))  f(s)}\psi'(s) ds.
	\end{align*}
	Using Proposistion \ref{Prop4.3} we get from above
	\begin{align*}
		\norm{((\Xi f )(a)-(\Xi f )(t_2))}
		\le & (\Psi(t_2,a))^{1-\gamma}\int_{a}^{t_2} (\Psi(t_2,s))^{\alpha-1}\frac{M}{\Gamma(\alpha)}\norm{f(s)}\psi'(s) ds\\
		\le &  \frac{M}{\Gamma(\alpha)} (\Psi(t_2,a))^{1-\gamma}\int_{a}^{t_2}\psi'(s) (\Psi(t_2,s))^{\alpha-1} \norm{f(s)} ds.
	\end{align*}
	By using H$\ddot{o}$lder's inequality we obtain
	\begin{align*}
		&\norm{((\Xi f )(a)-(\Xi f )(t_2))}\\
		\le & \frac{M}{\Gamma(\alpha)} (\Psi(t_2,a))^{1-\gamma}\norm{f}_{L^2([a,b],X)}\left(\int_{a}^{t_2}[\psi'(s) (\Psi(t_2,s))^{\alpha-1}]^2ds\right)^{\frac{1}{2}}\\
		\le &\frac{M}{\Gamma(\alpha)} (\Psi(t_2,a))^{1-\gamma}R\left(\int_{a}^{t_2}[\psi'(s) (\Psi(t_2,s))^{\alpha-1}]^2ds\right)^{\frac{1}{2}}.
	\end{align*}
	By means of Lemma \ref{tech}, we conclude that $	\norm{((\Xi f )(a)-(\Xi f )(t_2))}\to 0$ as $t_2\to a$.\\
	For  $a<t_1< t_2 \le b$ we have
	\begin{align*}
		&\norm{(\Xi f )(t_2)-(\Xi f )(t_1)}\\
		&=\norm{\int_{a}^{t_2}(\Psi(t_2,a))^{1-\gamma} (\Psi(t_2,s))^{\alpha-1} P_{\alpha}(\Psi(t_2,s))  f(s)\psi'(s) ds\right.\\
			&\left.\hspace{3cm}-\int_{a}^{t_1}(\Psi(t_1,a))^{1-\gamma} (\Psi(t_1,s))^{\alpha-1} P_{\alpha}(\Psi(t_1,s))  f(s)\psi'(s) ds}\\
		&=\norm{\int_{a}^{t_1}(\Psi(t_2,a))^{1-\gamma} (\Psi(t_2,s))^{\alpha-1} P_{\alpha}(\Psi(t_2,s))  f(s)\psi'(s) ds\right.\\
			&\left.\hspace{2cm}+\int_{t_1}^{t_2}(\Psi(t_2,a))^{1-\gamma} (\Psi(t_2,s))^{\alpha-1} P_{\alpha}(\Psi(t_2,s))  f(s)\psi'(s) ds\right.\\
			&\left.\hspace{3cm}-\int_{a}^{t_1}(\Psi(t_1,a))^{1-\gamma} (\Psi(t_1,s))^{\alpha-1} P_{\alpha}(\Psi(t_1,s))  f(s)\psi'(s) ds}\\
		&=\norm{\int_{a}^{t_1}(\Psi(t_1,a))^{1-\gamma}(\Psi(t_1,s))^{\alpha-1}[P_{\alpha}(\Psi(t_2,s))-P_{\alpha}(\Psi(t_1,s))]f(s)\psi'(s)ds\right.\\
			&\left.\hspace{2cm}-\int_{a}^{t_1}(\Psi(t_1,a))^{1-\gamma}(\Psi(t_1,s))^{\alpha-1}P_{\alpha}(\Psi(t_2,s))f(s)\psi'(s)ds\right.\\
			&\left.\hspace{2cm}+\int_{a}^{t_1} (\Psi(t_2,a))^{1-\gamma}(\Psi(t_2,s))^{\alpha-1} P_{\alpha}(\Psi(t_2,s))  f(s)\psi'(s) ds\right.\\
			&\left.\hspace{2cm}+\int_{t_1}^{t_2}(\Psi(t_2,a))^{1-\gamma} (\Psi(t_2,s))^{\alpha-1} P_{\alpha}(\Psi(t_2,s))  f(s)\psi'(s) ds}\\
		&\le\int_{a}^{t_1}(\Psi(t_1,a))^{1-\gamma}(\Psi(t_1,s))^{\alpha-1}\norm{[P_{\alpha}(\Psi(t_2,s))-P_{\alpha}(\Psi(t_1,s))]f(s)}\psi'(s)ds\\
		&\hspace{2cm}+\int_{a}^{t_1} [(\Psi(t_1,a))^{1-\gamma}(\Psi(t_1,s))^{\alpha-1}-(\Psi(t_2,a))^{1-\gamma}(\Psi(t_2,s))^{\alpha-1}]\\
		&\hspace{4cm}\times \norm{P_{\alpha}(\Psi(t_2,s))  f(s)}\psi'(s) ds\\
		&\hspace{2cm}+\int_{t_1}^{t_2}(\Psi(t_2,a))^{1-\gamma} (\Psi(t_2,s))^{\alpha-1} \norm{P_{\alpha}(\Psi(t_2,s))  f(s)}\psi'(s) ds.
	\end{align*}
	Using again Proposition \ref{Prop4.3} we obtain from above
	\begin{equation}
		\begin{aligned}
			&\norm{(\Xi f )(t_2)-(\Xi f )(t_1)}\\
			\le&\int_{a}^{t_1}(\Psi(t_1,a))^{1-\gamma}(\Psi(t_1,s))^{\alpha-1}\norm{[P_{\alpha}(\Psi(t_2,s))-P_{\alpha}(\Psi(t_1,s))]f(s)}\psi'(s)ds\\
			+&\frac{M}{\Gamma(\alpha)}\int_{a}^{t_1} [(\Psi(t_1,a))^{1-\gamma}(\Psi(t_1,s))^{\alpha-1}-(\Psi(t_2,a))^{1-\gamma}(\Psi(t_2,s))^{\alpha-1}]\norm{f(s)}\psi'(s) ds\\
			&\hspace{4cm}+\frac{M}{\Gamma(\alpha)}\int_{t_1}^{t_2}(\Psi(t_2,a))^{1-\gamma} (\Psi(t_2,s))^{\alpha-1} \norm{f(s)}\psi'(s) ds\\
			&=: I_1+I_2+I_3,
		\end{aligned}
	\end{equation}
	where
	\begin{equation*}
		I_1:=\int_{a}^{t_1}(\Psi(t_1,a))^{1-\gamma}(\Psi(t_1,s))^{\alpha-1}\norm{[P_{\alpha}(\Psi(t_2,s))-P_{\alpha}(\Psi(t_1,s))]f(s)}\psi'(s)ds,
	\end{equation*}
	\begin{align*}
		I_2:=&\frac{2M}{\Gamma(\alpha)}\int_{a}^{t_1} [(\Psi(t_1,a))^{1-\gamma}(\Psi(t_1,s))^{\alpha-1}-(\Psi(t_2,a))^{1-\gamma}(\Psi(t_2,s))^{\alpha-1}]\\
		&\hspace{5cm}\times \norm{f(s)}\psi'(s) ds
	\end{align*}
	and
	\begin{align*}
		I_3:=&\frac{M}{\Gamma(\alpha)}\int_{a}^{t_2}(\Psi(t_2,a))^{1-\gamma} (\Psi(t_2,s))^{\alpha-1} \norm{f(s)}\psi'(s) ds\\
		&\hspace{3cm}-\frac{M}{\Gamma(\alpha)}\int_{a}^{t_1}(\Psi(t_1,a))^{1-\gamma} (\Psi(t_1,s))^{\alpha-1} \norm{f(s)}\psi'(s) ds.
	\end{align*}
	
	As before (the case where $t_1=a$), one can prove that $I_3 \to 0$ as $t_2 \to t_1$.\\
	Since,
	\begin{equation*}
		\begin{aligned}
			&[(\Psi(t_1,a))^{1-\gamma}(\Psi(t_1,s))^{\alpha-1}-(\Psi(t_2,a))^{1-\gamma}(\Psi(t_2,s))^{\alpha-1}]  \norm{f(s)}\psi'(s)\\
			&\le (\Psi(t_1,a))^{1-\gamma}(\Psi(t_1,s))^{\alpha-1} \norm
			f(s)\psi'(s),
		\end{aligned}
	\end{equation*}
	and
	\begin{equation*}
		\int_{a}^{t_1} (\Psi(t_1,a))^{1-\gamma}(\Psi(t_1,s))^{\alpha-1} \norm{f(s)}\psi'(s)ds
	\end{equation*}
	exists, so applying the Lebesgue Dominated Convergence theorem, we can see  $I_2 \to 0$ as $t_2 \to t_1$.\\
	It remains to prove $I_1\to 0$ as $t_2\to t_1$. For $\epsilon>0$, the integral $I_1$ can be written as
	\begin{align*}
		I_1:=&\int_{a}^{t_1-\epsilon}(\Psi(t_1,a))^{1-\gamma}(\Psi(t_1,s))^{\alpha-1}\norm{[P_{\alpha}(\Psi(t_2,s))-P_{\alpha}(\Psi(t_1,s))]f(s)}\psi'(s)ds\\
		&+\int_{t_1-\epsilon}^{t_1}(\Psi(t_1,a))^{1-\gamma}(\Psi(t_1,s))^{\alpha-1}\norm{[P_{\alpha}(\Psi(t_2,s))-P_{\alpha}(\Psi(t_1,s))]f(s)}\psi'(s)ds\\
		\le&\int_{a}^{t_1-\epsilon}(\Psi(t_1,a))^{1-\gamma}(\Psi(t_1,s))^{\alpha-1}\norm{f(s)}\\
		&\hspace{2cm}\times\sup_{s \in [a,t_1-\epsilon]}\norm{P_{\alpha}(\Psi(t_2,s))-P_{\alpha}(\Psi(t_1,s))}_{B(X)}\psi'(s)ds
	\end{align*}
	\begin{align*}
		&+\frac{2M}{\Gamma(\alpha)}\int_{t_1-\epsilon}^{t_1}(\Psi(t_1,a))^{1-\gamma}(\Psi(t_1,s))^{\alpha-1}\norm{f(s)}\psi'(s)ds\\
		\le&\int_{a}^{t_1-\epsilon}(\Psi(t_1,a))^{1-\gamma}(\Psi(t_1,s))^{\alpha-1}\norm{f(s)}\\
		&\hspace{3cm}\times\sup_{s \in [a,t_1-\epsilon]}\norm{P_{\alpha}(\Psi(t_2,s))-P_{\alpha}(\Psi(t_1,s))}_{B(X)}\psi'(s)ds\\
		+&\frac{2M}{\Gamma(\alpha)}\left[\int_{a}^{t_1}(\Psi(t_1,a))^{1-\gamma}(\Psi(t_1,s))^{\alpha-1}\norm{f(s)}\psi'(s)ds\right.\\
		&\left.\hspace{1cm}-\int_{a}^{t_1-\epsilon}(\psi(t_1-\epsilon)-\psi(a))^{1-\gamma}(\psi(t_1-\epsilon)-\psi(s))^{\alpha-1}\norm{f(s)}\psi'(s)ds\right]\\
		+&\frac{2M}{\Gamma(\alpha)}\int_{a}^{t_1-\epsilon}\left[(\psi(t_1-\epsilon)-\psi(a))^{1-\gamma}(\psi(t_1-\epsilon)-\psi(s))^{\alpha-1}\right.\\
		&\left.\hspace{3cm}-(\Psi(t_1,a))^{1-\gamma}(\Psi(t_1,s))^{\alpha-1}\right]\norm{f(s)}\psi'(s)ds\\
		=: &J_1 + J_2 + J_3,
	\end{align*}
	
	where 
	\begin{align*}
		J_1\le& \sup_{s \in [a,t_1-\epsilon]}\norm{P_{\alpha}(\Psi(t_2,s))-P_{\alpha}(\Psi(t_1,s))}_{B(X)}\\
		&\int_{a}^{t_1-\epsilon}(\Psi(t_1,a))^{1-\gamma}(\Psi(t_1,s))^{\alpha-1}\norm{f(s)}\psi^{\prime}(s)ds,
	\end{align*}
	\begin{equation*}
		\begin{aligned}
			J_2:=&\frac{2M}{\Gamma(\alpha)}\left[\int_{a}^{t_1}(\Psi(t_1,a))^{1-\gamma}(\Psi(t_1,s))^{\alpha-1}\norm{f(s)}\psi'(s)ds\right.\\
			&\left.\hspace{1cm}-\int_{a}^{t_1-\epsilon}(\psi(t_1-\epsilon)-\psi(a))^{1-\gamma}(\psi(t_1-\epsilon)-\psi(s))^{\alpha-1}\norm{f(s)}\psi'(s)ds\right],
		\end{aligned}
	\end{equation*}
	and
	\begin{equation}
		\begin{aligned}
			J_3:=&\frac{2M}{\Gamma(\alpha)}\int_{a}^{t_1-\epsilon}\left[(\psi(t_1-\epsilon)-\psi(a))^{1-\gamma}(\psi(t_1-\epsilon)-\psi(s))^{\alpha-1}\right.\\
			&\left.\hspace{3cm}-(\Psi(t_1,a))^{1-\gamma}(\Psi(t_1,s))^{\alpha-1}\right]\norm{f(s)}\psi'(s)ds.
		\end{aligned}
	\end{equation}
	It is clear that $J_1 \to 0$ as $t_2 \to t_1$.\\
	Also repeating the same process which we use to show $I_1$ and $I_2$ goes to zero, we can conclude that $J_2$ and $J_3$ goes to zero as $\epsilon \to 0$ and hence $I_3 \to 0 $ as $t_2 \to t_1$.\\
	Thus$$\norm{(\Xi f )(t_2)-(\Xi f )(t_1)} \to 0,$$ independently of $f  \in B_R$ as $t_2 \to t_1$, which means  $\{\Xi f  : f  \in B_R \}$ is equicontinuous and consequently 
	$\{\Xi : f  \in  B_R \}$ is equicontinuous.\\

	We now show that the set $\{(\Psi(t,a))^{1-\gamma}\Lambda f(t): f\in B_R\}\subset C([a,b],X)$  is relatively compact in $X$. For $t=a$, it is trivial. Take $a<t\le b$ and choose $a<\lambda<t$, then we define the operator $\Gamma_{\lambda,\delta}:L^2([a,b],X)\to C([a,b],X)$ as follows:
	\begin{align*}
		&	\Gamma_{\lambda,\delta}(f)(t)\\
		&=(\Psi(t,a))^{1-\gamma}  \int_{a}^{t-\lambda}  \int_{\delta}^{\infty} \psi'(s) \alpha \theta  M_{\alpha}(\theta)(\Psi(t,s))^{\alpha-1}  T((\Psi(t,s))^{\alpha}\theta) f(s)d\theta ds, t\in [a,b].
	\end{align*}
	We can express $\Gamma_{\lambda,\delta}$ as follows:
	\begin{align*}
		&\Gamma_{\lambda,\delta}(f)(t)\\
		=& (\Psi(t,a))^{1-\gamma} T(\epsilon^{\alpha}\delta) \int_{a}^{t-\lambda}  \int_{\delta}^{\infty} \psi'(s) \alpha \theta  M_{\alpha}(\theta)(\Psi(t,s))^{\alpha-1} \\
		&\hspace{7cm}\times T((\Psi(t,s))^{\alpha}\theta-\epsilon^{\alpha}\delta) f(s)d\theta ds.
	\end{align*}
	It is easy to see that the set
	\begin{equation}
		\left\{ \int_{a}^{t-\lambda}  \int_{\delta}^{\infty} \psi'(s) \alpha \theta  M_{\alpha}(\theta)(\Psi(t,s))^{\alpha-1}  T((\Psi(t,s))^{\alpha}\theta-\epsilon^{\alpha}\delta) f(s)d\theta ds : f\in B_R \right\}
	\end{equation}
	is bounded. Indeed,
	\begin{align*}
		&\norm{ \int_{a}^{t-\lambda}  \int_{\delta}^{\infty} \psi'(s) \alpha \theta  M_{\alpha}(\theta)(\Psi(t,s))^{\alpha-1}  T((\Psi(t,s))^{\alpha}\theta-\epsilon^{\alpha}\delta) f(s)d\theta ds }\\
		\le & \int_{a}^{t-\lambda}  \int_{\delta}^{\infty} \psi'(s) \alpha \theta  M_{\alpha}(\theta)(\Psi(t,s))^{\alpha-1} \norm{T((\Psi(t,s))^{\alpha}\theta-\epsilon^{\alpha}\delta) f(s)}ds\\
		\le &M\int_{a}^{t-\lambda}  \int_{\delta}^{\infty} \psi'(s) \alpha \theta  M_{\alpha}(\theta)(\Psi(t,s))^{\alpha-1} \norm{f(s)}ds,
	\end{align*}
	which is finite.\\
	Hence, the set
	\begin{equation}
		\left\{ \int_{a}^{t-\lambda}  \int_{\delta}^{\infty} \psi'(s) \alpha \theta  M_{\alpha}(\theta)(\Psi(t,s))^{\alpha-1}  T((\Psi(t,s))^{\alpha}\theta-\epsilon^{\alpha}\delta) f(s)d\theta ds : f\in B_R \right\}
	\end{equation}
	bounded.\\
	By the compactness of $T(t)$, we get the set 
	\begin{equation}
		\left\{(\Gamma _{\lambda,\delta}f)(t) :f \in B_R    \right\}
	\end{equation}
	is relatively compact in $X$ for all $\lambda>0$ and all $\delta>0$.\\
	Furthermore, for any $f \in B_R $ we have
	\begin{align*}
		&\norm{(\Gamma f)(t)-(\Gamma _{\lambda,\delta}f)(t)}\\
		=&(\Psi(t,a))^{1-\gamma}\norm{\left[  \int_{a}^{t}  \int_{0}^{\infty} \alpha \theta  M_{\alpha}(\theta)  T((\Psi(t,s))^{\alpha}\theta)  (\psi(t-\psi(s))^{\alpha-1}   f(s)\psi'(s)d\theta ds\right]\right.\\
			&\left.\hspace{2cm}-\left[  \int_{a}^{t-\epsilon}  \int_{\delta}^{\infty} \alpha \theta  M_{\alpha}(\theta)  T((\Psi(t,s))^{\alpha}\theta)  (\Psi(t,s))^{\alpha-1}   f(s)\psi'(s)d\theta ds\right]}\\
		=&(\Psi(t,a))^{1-\gamma}\norm{ \int_{a}^{t}  \int_{0}^{\infty} \alpha \theta  M_{\alpha}(\theta)  T((\Psi(t,s))^{\alpha}\theta)  (\Psi(t,s))^{\alpha-1}   f(s)\psi'(s)d\theta ds\right.\\
			&\left.\hspace{2cm}-  \int_{a}^{t-\epsilon}  \int_{\delta}^{\infty} \alpha \theta  M_{\alpha}(\theta)  T((\Psi(t,s))^{\alpha}\theta)  (\Psi(t,s))^{\alpha-1}   f(s)\psi'(s)d\theta ds}\\
	\end{align*}
	\begin{align*}
		& = (\Psi(t,a))^{1-\gamma}\norm{ \int_{a}^{t}  \int_{0}^{\delta} \alpha \theta  M_{\alpha}(\theta)  T((\Psi(t,s))^{\alpha}\theta)  (\Psi(t,s))^{\alpha-1}   f(s)\psi'(s)d\theta ds\right.\\
			&\left.+\int_{a}^{t}  \int_{\delta}^{\infty} \alpha \theta  M_{\alpha}(\theta)  T((\Psi(t,s))^{\alpha}\theta)  (\Psi(t,s))^{\alpha-1}   f(s)\psi'(s)d\theta ds\right.\\
			&\left.-  \int_{a}^{t-\epsilon}  \int_{\delta}^{\infty} \alpha \theta  M_{\alpha}(\theta)  T((\Psi(t,s))^{\alpha}\theta)  (\Psi(t,s))^{\alpha-1}   f(s)\psi'(s)d\theta ds}\\
		& = (\Psi(t,a))^{1-\gamma}\norm{\int_{a}^{t}  \int_{0}^{\delta} \alpha \theta  M_{\alpha}(\theta)(\Psi(t,s))^{\alpha-1}   T((\Psi(t,s))^{\alpha}\theta)    f(s)\psi'(s)d\theta ds\right.\\
			&\left.+(\Psi(t,a))^{1-\gamma} \int_{t-\epsilon}^{t}  \int_{\delta}^{\infty} \alpha \theta  M_{\alpha}(\theta)  T((\Psi(t,s))^{\alpha}\theta)  (\Psi(t,s))^{\alpha-1}   f(s)\psi'(s)d\theta ds}\\
		&\le (\Psi(t,a))^{1-\gamma} \int_{a}^{t}  \int_{0}^{\delta} \alpha \theta  M_{\alpha}(\theta)(\Psi(t,s))^{\alpha-1} \norm{T((\Psi(t,s))^{\alpha}\theta)    f(s)}\psi'(s)d\theta ds\\
		&\hspace{1cm}+(\Psi(t,a))^{1-\gamma}\int_{t-\epsilon}^{t}  \int_{\delta}^{\infty} \alpha \theta  M_{\alpha}(\theta) \norm{T((\Psi(t,s))^{\alpha}\theta)  (\Psi(t,s))^{\alpha-1}   f(s)}\psi'(s)d\theta ds\\
		&\le (\Psi(t,a))^{1-\gamma} M\int_{a}^{t}  \int_{0}^{\delta} \alpha \theta  M_{\alpha}(\theta)(\Psi(t,s))^{\alpha-1}\norm{f(s)}\psi'(s)d\theta ds\\
		& +M(\Psi(t,a))^{1-\gamma} \int_{t-\epsilon}^{t}  \int_{\delta}^{\infty} \alpha \theta  M_{\alpha}(\theta)(\Psi(t,s))^{\alpha-1}\norm{f(s)}\psi'(s)d\theta ds\\
		& \le (\Psi(t,a))^{1-\gamma} M\int_{a}^{t}  \int_{0}^{\infty} \alpha \theta  M_{\alpha}(\theta)(\Psi(t,s))^{\alpha-1}\norm{f(s)}\psi'(s)d\theta ds\\
		&+M(\Psi(t,a))^{1-\gamma} \int_{t-\epsilon}^{t}  \int_{0}^{\infty} \alpha \theta  M_{\alpha}(\theta)\norm{f(s)}(\Psi(t,s))^{\alpha-1}\psi'(s)d\theta ds\\
		&\le (\Psi(t,a))^{1-\gamma} \frac{M}{\Gamma(\alpha)}\int_{a}^{t}(\Psi(t,s))^{\alpha-1}\norm{f(s)}\psi'(s) ds \\
		&+\frac{M}{\Gamma(\alpha)}(\Psi(t,a))^{1-\gamma} \int_{t-\epsilon}^{t}\norm{f(s)}(\Psi(t,s))^{\alpha-1}\psi'(s) ds \to 0 ~as~ \lambda
		,\delta \to 0.
	\end{align*}

	Therefore, there are relatively compact sets arbitrarily close to the set  $$\{(\Psi(t,a))^{1-\gamma}\Lambda f(t): f\in B_R\}\subset C([a,b], X).$$ Hence, the set $\{(\Psi(t,a))^{1-\gamma}\Lambda f(t): f\in B_R\}\subset C([a,b],X)$ is relatively compact in $X$.
	
	Combining these two with the Arzela-Ascoli theorem, the set $$\{(\psi(\cdot)-\psi(a))^{1-\gamma}\Lambda f(\cdot): f\in B_R\}\subset C([a,b],X)$$ is  relatively compact in $C([a,b],X)$. Consequently, the set $\{\Lambda f(\cdot): f\in B_R\}$ is relatively compact in $C^{1-\gamma;\psi}([a,b],X)$.
\end{proof}
\begin{corollary}\label{Corollary}
	Suppose $\{f_n\}\subset L^1([a,b],X)$ be an integrably bounded sequence, that means there exists $m\in L^{\frac{1}{r}}([a,b],\mathbb{R}^+)$, $r\in (0,\alpha)$ such that
	\begin{equation}
		\norm{f_n(t)}\le m(t)~\text{a.a.}~t\in [a,b].
	\end{equation}
	Then the set $\left\{\int_{a}^{b}K_{\alpha}(\Psi(t,s))f_n(s)\psi^{\prime}(s)ds\right\}$ is relatively compact in $C^{1-\gamma;\psi}([a,b],X)$.
\end{corollary}
We now give the following lemma which is of fundamental importance.
\begin{lemma}\label{Lem 5.6}
	Suppose that $h$ satisfies Hypothesis (H1-(H5). Then, we have
	\begin{itemize}
		\item[(1)] for each $(q(\cdot),u(\cdot))\in C^{1-\gamma;\psi}([a,b],X)\times L^2([a,b],Y)$, the inequality holds
		\begin{align*}
			J(q(\cdot), u(\cdot))\ge &-(\norm{k_1}_{L^1([a,b],\mathbb{R}^+)}+K\norm{q(\cdot)}_{C^{1-\gamma;\psi}([a,b],X)}\norm{k_2}_{L^2([a,b],\mathbb{R}^+)}\\
			&+c_h\sqrt{b}\norm{u(\cdot)}_{L^2([a,b],Y)}).
		\end{align*}
		\item[(2)] $(q(\cdot), u(\cdot))\mapsto J(q(\cdot), u(\cdot))$ is strongly weakly lower semicontinuous;
		\item[(3)] $u(\cdot)\mapsto J(q(\cdot), u(\cdot))$ is convex.
	\end{itemize}
\end{lemma}
\begin{lemma}\label{Lem 5.7}
	If $(q(\cdot), u(\cdot))\in C^{1-\gamma;\psi}([a,b], X)\times L^2([a,b], Y)$ is a solution of problem \eqref{4.1}-\eqref{4.2}, then the following inequality holds
	\begin{equation}
		\norm{q}_{C^{1-\gamma;\psi}([a,b],X)}\le N_0, \norm{u}_{L^2([a,b],Y)}\le N_0.
	\end{equation}
\end{lemma}
\begin{proof}
	If $(q(\cdot), u(\cdot))\in C^{1-\gamma;\psi}([a,b],X)\times L^2([a,b],Y)$ is a solution of problem \eqref{4.1}-\eqref{4.2}, then we have
	\begin{equation}
		q(t)=S_{\alpha,\beta}(\Psi(t,a))x_0+\int_{a}^{t}\psi^{\prime}(s)(\Psi(t,s))^{\alpha-1}P_{\alpha}(\Psi(t,s))Bu(s)ds, t\in [a,b].
	\end{equation}
	Using the expression \eqref{4.3} we compute
	\begin{align*}
		&\norm{(\Psi(t,a))^{1-\gamma}q(t)}_X\\
		=&\norm{(\Psi(t,a))^{1-\gamma}\left[S_{\alpha,\beta}(\Psi(t,a))x_0+\int_{a}^{t}\psi^{\prime}(s)(\Psi(t,s))^{\alpha-1}P_{\alpha}(\Psi(t,s))Bu(s)ds\right]}\\
		\le &(\Psi(t,a))^{1-\gamma}\left[\norm{S_{\alpha,\beta}(\Psi(t,a))x_0}+\norm{\int_{a}^{t}\psi^{\prime}(s)(\Psi(t,s))^{\alpha-1}P_{\alpha}(\Psi(t,s))Bu(s)ds}\right].
	\end{align*}
	Invoking Proposition \ref{Prop4.3} we obtain
	\begin{equation}
		(\Psi(t,a))^{1-\gamma}\norm{S_{\alpha,\beta}(\Psi(t,a))x_0}\le \frac{M}{\Gamma(\gamma)}\norm{x_0}.
	\end{equation}
	Also, by means of Proposition \ref{Prop4.3} we get
	\begin{align}
		\norm{\int_{a}^{t}\psi^{\prime}(s)(\Psi(t,s))^{\alpha-1}P_{\alpha}(\Psi(t,s))Bu(s)ds}\\
		\le &\frac{M}{\Gamma(\alpha)}\norm{B} \int_{a}^{t}\psi^{\prime}(s)(\Psi(t,s))^{\alpha-1}\norm{u(s)}ds.
	\end{align}
	Therefore, 
	\begin{align*}
		&\norm{(\Psi(t,a))^{1-\gamma}q(t)}\le \frac{M}{\Gamma(\gamma)}\norm{x_0}\\
		&+\frac{M}{\Gamma(\alpha)}\norm{B}(\Psi(t,a))^{1-\gamma} \int_{a}^{t}\psi^{\prime}(s)(\Psi(t,s))^{\alpha-1}\norm{u(s)}ds\\
		\le& \frac{M}{\Gamma(\gamma)}\norm{x_0}+\frac{MK}{\Gamma(\alpha)}\norm{B}(\Psi(t,a))^{1-\gamma}\norm{u}_{L^2([a,b],Y)} \left(\int_{a}^{t}\psi^{\prime}(s)(\Psi(t,s))^{2(\alpha-1)}ds\right)^{\frac{1}{2}}, 
	\end{align*}
	where $K=\sup_{t \in [a,b]}\abs{\psi^{\prime}(t)}$ and we use H$\ddot{o}$lder's inequality in the last inequality. After a simple calculation, we can write
	\begin{equation}\label{4.24}
		\norm{(\Psi(t,a))^{1-\gamma}q(t)}	\le \frac{M}{\Gamma(\gamma)}\norm{x_0}+\frac{MK}{\Gamma(\alpha)}\norm{B}(\Psi(t,a))^{1-\gamma}\norm{u} \left(\frac{(\Psi(t,a))^{2\alpha-1}}{2\alpha-1}\right)^{\frac{1}{2}}.
	\end{equation}
	Noting that $\norm{u(t)}\le a_U(t)+c_U(\Psi(t,a))^{1-\gamma}\norm{q(t)}$, by Hypothesis (U3). Using this fact, we obtain from \eqref{4.24}
	\begin{align*}
		&\norm{(\Psi(t,a))^{1-\gamma}q(t)}\\
		&\le \frac{M}{\Gamma(\gamma)}\norm{x_0}+\frac{M}{\Gamma(\alpha)}\norm{B}(\Psi(t,a))^{1-\gamma} \int_{a}^{t}\psi^{\prime}(s)(\Psi(t,s))^{\alpha-1}\norm{u(s)}ds\\
		\le& \frac{M}{\Gamma(\gamma)}\norm{x_0}\\
		&+\frac{M}{\Gamma(\alpha)}\norm{B}(\Psi(t,a))^{1-\gamma}\int_{a}^{t}\psi^{\prime}(s)(\Psi(t,s))^{\alpha-1}(a_U(s)+c_U(\Psi(s,a))^{1-\gamma}\norm{q(s)})ds.
	\end{align*}
	Again, by using H$\ddot{o}$lder's inequality and after a simple manipulation, we obtain
	\begin{align*}
		&\norm{(\Psi(t,a))^{1-\gamma}q(t)}\le \frac{M}{\Gamma(\gamma)}\norm{x_0}\\
		&+\frac{M}{\Gamma(\alpha)}\norm{B}(\Psi(t,a))^{1-\gamma}\norm{a_U}_{L^2([a,b],\mathbb{R}^+)} \left(\frac{(\Psi(t,a))^{2\alpha-1}}{2\alpha-1}\right)^{\frac{1}{2}}\\
		&+\frac{Mc_U}{\Gamma(\alpha)}\norm{B}(\Psi(t,a))^{1-\gamma}\int_{a}^{t}\psi^{\prime}(s)(\Psi(t,s))^{\alpha-1}(\psi(s)-\psi(a))^{1-\gamma}\norm{q(s)}ds.
	\end{align*}
	Let $y(t)=(\Psi(t,a))^{1-\gamma}q(t)$, then we have
	\begin{equation}\label{4.25}
		\norm{y(t)}\le C+D\int_{a}^{t}\psi^{\prime}(s)(\Psi(t,s))^{\alpha-1}\norm{y(s)}ds, t\in [a,b],
	\end{equation}
	where
	\begin{equation}
		C=\frac{M}{\Gamma(\gamma)}\norm{x_0}+\frac{M}{\Gamma(\alpha)}\norm{B}(\Psi(b,a))^{1-\gamma}\norm{a_U}_{L^2([a,b],\mathbb{R}^+)} \left(\frac{(\Psi(b,a))^{2\alpha-1}}{2\alpha-1}\right)^{\frac{1}{2}}
	\end{equation} 
	and
	\begin{equation}
		D=\frac{Mc_U}{\Gamma(\alpha)}\norm{B}(\Psi(b,a))^{1-\gamma}.
	\end{equation}
	By Gronwall's inequality, we conclude from \eqref{4.25}
	\begin{equation}
		\norm{y(t)}\le CE_{\alpha}(D\Gamma(\alpha)(\Psi(t,a))^{\alpha}).
	\end{equation}
	Consequently, we obtain
	\begin{equation}
		\norm{y(t)}\le CE_{\alpha}(D\Gamma(\alpha)(\Psi(b,a))^{\alpha})=M_0~\text{say}~, t\in [a,b].
	\end{equation}
	Hence,
	\begin{equation}\label{4.26}
		\norm{q}_{C^{1-\gamma;\psi}([a,b],X)}\le M_0.
	\end{equation}
	Again, by Hypothesis (U3), we obtain
	\begin{equation}
		\norm{u(t)}\le a_U(t)+c_U(\Psi(t,a))^{1-\gamma}\norm{q(t)}, t\in [a,b].
	\end{equation}
	By means of \eqref{4.26} we obtain
	\begin{equation}
		\norm{u(t)}\le a_U(t)+c_U\norm{q}_{C^{1-\gamma;\psi}([a,b],X)}\le a_U(t)+c_UM_0, t\in [a,b].
	\end{equation}
	By H$\ddot{o}$lder's inequality we obtain
	\begin{equation}
		\norm{u}_{L^2([a,b],Y)}\le N_0, ~\text{for some}~N_0>0.
	\end{equation}
	Let $K_0=\max\{M_0, N_0\}$. Then we obtain
	\begin{equation}
		\norm{q}_{C^{1-\gamma;\psi}([a,b],X)}\le N_0, \norm{u}_{L^2([a,b],Y)}\le N_0.
	\end{equation}
\end{proof}
\begin{theorem}
	Assume Hypotheses H(T), H(U), and H(h) hold. Then Problem 1 has a solution $(q(\cdot),u(\cdot))\in C^{1-\gamma;\psi}([a,b],X)\times L^2([a,b],Y)$.
\end{theorem}
\begin{proof}
	Applying Lemma \ref{Lem 5.6}, we can see that $J$ is bounded from below on $R_U(x_0)$. Let $\{(q_n(\cdot), u_n(\cdot))\}\subset R_U(x_0)$ be a minimizing sequence of Problem 1, that is.
	\begin{equation}
		\inf_{(q(\cdot),u(\cdot))\in R_U(x_0)} J(q(\cdot), u(\cdot))=\lim_{n\to \infty} J(q_n(\cdot),u_n(\cdot)).
	\end{equation}
	For each $n\in \mathbb{N}$, we have
	\begin{equation}
		\begin{cases}
			^{H}D^{\alpha,\beta;\psi}_{a+} q_n(t)=Aq_n(t)+Bu_n(t), t\in [a,b]\\
			(I_{a+}^{1-\gamma;\psi}q_n)(a)=x_0,
		\end{cases}
	\end{equation}
	where $u_n(t)\in U(t,q_n(t))$ for a.e. $t\in [a,b]$.  Thus
	\begin{equation}\label{4.37}
		q_n(t)=S_{\alpha,\beta}(\Psi(t,a))x_0+\int_{a}^{t}\psi^{\prime}(s)(\Psi(t,s))^{\alpha-1}P_{\alpha}(\Psi(t,s))Bu_n(s)ds, t\in [a,b].
	\end{equation}
	The boundedness of $\{u_n(\cdot)\}$ and the fact that $B$ is bounded linear, allow us to assume that
	\begin{equation}
		u_n\to u~\text{weakly in}~L^2([a,b],Y)~\text{and}~Bu_n\to Bu~\text{weakly in}~L^2([a,b],X).
	\end{equation}
	Denote by $y_n(t)=(\Psi(t,a))^{1-\gamma}q_n(t)$. The set $$\{y\in C([a,b],X): y(t)=(\Psi(t,a))^{1-\gamma}q(t), q\in C^{1-\gamma;\psi}([a,b],X)\}$$ is equicontinuous and relatively compact subset of $C([a,b],X)$. Hence we can deduce that $\{q_n(\cdot)\}$ is relatively compact on $C^{1-\gamma;\psi}([a,b],X)$. Therefore, there exists a function $q_0\in C^{1-\gamma;\psi}([a,b],X)$ such that
	\begin{equation}
		q_{n_k}\to q_0~\text{as}~k\to \infty~\text{in}~C^{1-\gamma;\psi}([a,b],X).
	\end{equation} 
	Therefore, by passing limit in \eqref{4.37} and invoking Lemma \ref{Lem 5.8} we obtain
	\begin{equation}
		q_n(t)\to q(t)= S_{\alpha,\beta}(\Psi(t,a))x_0+\int_{a}^{t}\psi^{\prime}(s)(\Psi(t,s))^{\alpha-1}P_{\alpha}(\Psi(t,s))Bu(s)ds, t\in [a,b].
	\end{equation}
	By the fact that $u_n\to u$ weakly in $L^2([a,b],Y)$, Mazur's Theorem implies
	\begin{equation}\label{4.41}
		u(t)\in \bigcap_{n=1}^{\infty}\overline{conv}~\left(\bigcup_{k=n}^{\infty}u_k(t)\right)~\text{a.e.}~t\in [a,b].
	\end{equation} 
	Thanks to Hypothesis (U4), the map $x\to U(t,x)$ is Hausdorff continuous a.e. $t\in [a,b]$. Then Proposition 1.2.86 in \cite{papageorgiou1997handbook} implies, the map $x\to U(t,x)$ has property Q a.e. $t\in [a,b]$. Hence we have
	\begin{equation}\label{4.42}
		\bigcap_{n=1}^{\infty}\overline{conv}~\left(\bigcup_{k=n}^{\infty}u_k(t)\right)\subset \overline{conv}~U(t,x_0(t)) ~\text{for a.e.}~t\in [a,b].
	\end{equation}
	According to \eqref{4.41} and \eqref{4.42} we obtain that $u(t)\in U(t,q_0(t))$ a.e. $t\in [a,b]$. 
	
	This shows that $q(\cdot)\in C^{1-\gamma;\psi}([a,b],X)$ is a solution of Problem \eqref{4.1} and hence $(q(\cdot), u(\cdot))\in R_U(x_0)$.

	We will prove that $(q(\cdot),u(\cdot))$ is a solution to Problem 1. We compute
	\begin{align*}
		J=&\inf_{(q(\cdot),u(\cdot))\in R_U(x_0)} J(q(\cdot),u(\cdot))=\lim_{n\to \infty} J(q_n(\cdot),u_n(\cdot))=\lim_{n\to \infty} \int_{a}^{b}h(t,q_n(t),u_n(t))dt\\
		\ge &\int_{a}^{b}h(t,q(t),u(t))dt= J(q(\cdot),u(\cdot))\ge \inf_{(q(\cdot),u(\cdot))\in R_U(x_0)} J(q(\cdot),u(\cdot))=J.
	\end{align*}
	This shows that $(q(\cdot), u(\cdot))$ is a solution to Problem 1.
\end{proof}
We now consider the optimal control problem for minimizing the cost functional $J: C^{1-\gamma;\psi}[a,b]\times L^2([a,b], Y)\to \mathbb{R}$ given by
\begin{equation}\label{4.43}
	J(q(\cdot),u(\cdot))=\norm{q(b)-x_b}_X^2+\lambda \int_{a}^{b}\norm{u(t)}_U^2dt,
\end{equation}
where $x_b\in X, \lambda>0$ and $q(\cdot)$ is the unique mild solution of the fractional linear control problem involving $\psi$-Hilfer fractional derivative given by
\begin{equation}
	\begin{cases}
		^{H}D^{\alpha,\beta;\psi}_{a+} q(t)=Aq(t)+Bu(t), t\in [a,b]\\
		(I_{a+}^{1-\gamma;\psi}q)(a)=x_0.
	\end{cases}
\end{equation}
Since $Bu(\cdot)\in L^1([a,b],X)$, the system \eqref{4.1} has a unique mild solution\\ $q(\cdot)\in C^{1-\gamma;\psi}[a,b]$ given by
\begin{equation}
	q(t)=S_{\alpha,\beta}(\Psi(t,a))x_0+\int_{a}^{t}\psi^{\prime}(s)(\Psi(t,s))^{\alpha-1}P_{\alpha}(\Psi(t,s))Bu(s)ds, t\in [a,b],
\end{equation}
for any $u\in L^2([a,b],Y)$. Let
\begin{equation}
	A_{ad}=\{(q(\cdot), u(\cdot)): q(\cdot)~\text{is the unique mild solution of} ~\eqref{4.1}~\text{ with}~u(\cdot)\in L^2([a,b],Y )\}
\end{equation}
be the admissible class for the system \eqref{4.1}. Since we know that, for any given control $u(\cdot)\in L^2([a,b], Y)$, the system \eqref{4.1} has a unique mild solution, which implies that the set $A_{ad}$ is non-empty. We now formulate the optimal control problem as
\begin{equation}\label{4.47}
	\min_{(q(\cdot),u(\cdot))\in A_{ad}} J(q(\cdot),u(\cdot)).
\end{equation}
\begin{theorem}
	For a given $x_0\in X$ and fixed $\frac{1}{2}<\alpha\le1$, there exists a unique optimal pair $(q^0(\cdot), u^0(\cdot))\in A_{ad}$ for the problem \eqref{4.47}.
\end{theorem}
\begin{proof}
	Let us assume 
	\begin{equation}
		J=\inf_{u(\cdot)\in L^2([a,b],Y)} J(q(\cdot),u(\cdot)).
	\end{equation}
	Since $0\le J<\infty$, there exists a minimizing sequence $\{u^n(\cdot)\}_{n\ge 1}\in L^2([a,b],Y)$ such that
	\begin{equation}
		\lim_{n\to \infty} J(q^n(\cdot), u^n(\cdot))=J.
	\end{equation}
	Here $q^n(\cdot)$ denotes the unique mild solution of the system \eqref{4.1} with the control $u^n(\cdot)$ for each $n\in \mathbb{N}$ with $I_{a+}^{1-\gamma;\psi} q^n(a)=x_0$. Thus
	\begin{equation}\label{4.50}
		q^n(t)=S_{\alpha,\beta}(\Psi(t,a))x_0+\int_{a}^{t}\psi^{\prime}(s)(\Psi(t,s))^{\alpha-1}P_{\alpha}(\Psi(t,s))Bu^n(s)ds, t\in [a,b].
	\end{equation}
	Since $0\in L^2([a,b],Y)$, without loss of generality, we may assume that $J(q^n(\cdot), u^n(\cdot))\le J(q(\cdot),0)$, where $(q(\cdot),0)\in A_{ad}$. By using the definition of $J(\cdot,\cdot)$, we easily estimate
	\begin{equation}
		\norm{q^n(b)-x_b}^2+\lambda\int_{a}^{b}\norm{u^n(t)}_U^2dt\le \norm{q(b)-x_b}_X^2\le 2(\norm{q(b)}^2+\norm{x_b}^2)<\infty.
	\end{equation}
	From the above fact, it is clear that there exists a large $L>0$ (independent of $n$) such that
	\begin{equation}\label{4.52}
		\int_{a}^{b}\norm{u^n(t)}_U^2dt\le L<\infty.
	\end{equation} 
	Using the expression \eqref{4.50} we compute
	\begin{align*}
		&\norm{(\Psi(t,a))^{1-\gamma}q^n(t)}_X\\
		=&\norm{(\Psi(t,a))^{1-\gamma}\left[S_{\alpha,\beta}(\Psi(t,a))x_0+\int_{a}^{t}\psi^{\prime}(s)(\Psi(t,s))^{\alpha-1}P_{\alpha}(\Psi(t,s))Bu^n(s)ds\right]}\\
		\le &(\Psi(t,a))^{1-\gamma}\left[\norm{S_{\alpha,\beta}(\Psi(t,a))x_0}+\norm{\int_{a}^{t}\psi^{\prime}(s)(\Psi(t,s))^{\alpha-1}P_{\alpha}(\Psi(t,s))Bu^n(s)ds}\right]
	\end{align*}
	Now,
	\begin{equation}
		(\Psi(t,a))^{1-\gamma}\norm{S_{\alpha,\beta}(\Psi(t,a))x_0}\le \frac{M}{\Gamma(\gamma)}\norm{x_0}.
	\end{equation}
	Also,
	\begin{align*}
		&	\norm{\int_{a}^{t}\psi^{\prime}(s)(\Psi(t,s))^{\alpha-1}P_{\alpha}(\Psi(t,s))Bu^n(s)ds}\\
		\le &\int_{a}^{t}\psi^{\prime}(s)(\Psi(t,s))^{\alpha-1}\norm{P_{\alpha}(\Psi(t,s))Bu^n(s)}ds\\
		\le &\frac{M}{\Gamma(\alpha)}\norm{B} \int_{a}^{t}\psi^{\prime}(s)(\Psi(t,s))^{\alpha-1}\norm{u^n(s)}ds\\
	\end{align*}
	Therefore,
	\begin{align*}
		&\norm{(\Psi(t,a))^{1-\gamma}q^n(t)}\le \frac{M\norm{x_0}}{\Gamma(\gamma)}\\
		&\hspace{2cm}+\frac{M\norm{B}}{\Gamma(\alpha)}(\Psi(t,a))^{1-\gamma} \int_{a}^{t}\psi^{\prime}(s)(\Psi(t,s))^{\alpha-1}\norm{u^n(s)}ds\\
		\le& \frac{M}{\Gamma(\gamma)}\norm{x_0}+\frac{MK}{\Gamma(\alpha)}\norm{B}(\Psi(t,a))^{1-\gamma}\norm{u^n}_{L^2([a,b],Y)} \left(\int_{a}^{t}\psi^{\prime}(s)(\Psi(t,s))^{2(\alpha-1)}ds\right)^{\frac{1}{2}}\\
		\le& \frac{M}{\Gamma(\gamma)}\norm{x_0}+\frac{MK}{\Gamma(\alpha)}\norm{B}(\Psi(t,a))^{1-\gamma}\norm{u^n}_{L^2([a,b],Y)} \left(\frac{(\Psi(t,a))^{2\alpha-1}}{2\alpha-1}\right)^{\frac{1}{2}}.
	\end{align*}
	In the above expression $K=\sup_{t\in [a,b]}\abs{\psi^{\prime}(t)}$. Therefore we have
	\begin{equation}
		\norm{q^n}_{C^{1-\gamma;\psi}[a,b]}<\infty, ~\text{for}~ \frac{1}{2}<\alpha\le1.
	\end{equation}
	Denote by $y_n(t)=(\Psi(t,a))^{1-\gamma}q^n(t)$.  We can deduce that $\{q^n(\cdot)\}$ is relatively compact on $C^{1-\gamma;\psi}([a,b],X)$. Therefore, there exists a function $q^0\in C^{1-\gamma;\psi}([a,b],X)$ such that
	\begin{equation}
		q^{n_k}(\cdot)\to q^0(\cdot)~\text{as}~k\to \infty.
	\end{equation} 
	The estimate \eqref{4.52} implies that the sequence $\{u^n(\cdot)\}_{n\ge 1}$ is uniformly bounded in $L^2([a,b],Y)$. By the application of Banach Alaouglu theorem, we always find a subsequence, say $\{u^{n_k}(\cdot)\}$ of $\{u^n(\cdot)\}$ such that
	\begin{equation}
		u^{n_k}(\cdot)\to u^0(\cdot) ~\text{in}~L^2([a,b],Y), ~\text{as}~k\to \infty.
	\end{equation}
	Since the operator $B$ is bounded from $Y$ to $X$, then we have
	\begin{equation}
		Bu^{n_k}\to Bu^0~\text{in}~L^2([a,b],X)~\text{as}~k\to \infty.
	\end{equation}
	Moreover, by using the above convergences together with the compactness of the operator $Q: L^2([a,b], X)\to C^{1-\gamma;\psi}([a,b],X)$, given by
	\begin{equation}
		(Qf)(t)=\int_{a}^{t}\psi^{\prime}(s)(\Psi(t,s))^{\alpha-1}P_{\alpha}(\Psi(t,s))f(s)ds, t\in [a,b],
	\end{equation}
	we obtain
	\begin{align*}
		&(\Psi(t,a))^{1-\gamma}\norm{\int_{a}^{t}\psi^{\prime}(s)(\Psi(t,s))^{\alpha-1}P_{\alpha}(\Psi(t,s))Bu^{n_k}(s)ds\right.\\
			&\left.-\int_{a}^{t}\psi^{\prime}(s)(\Psi(t,s))^{\alpha-1}P_{\alpha}(\Psi(t,s))Bu^0(s)ds}\to 0~\text{as}~k\to \infty, \forall t\in [a,b].
	\end{align*}
	We now estimate
	\begin{align*}
		&(\Psi(t,a))^{1-\gamma}\norm{q^{n_k}(t)-q^*(t)}\\
		=&(\Psi(t,a))^{1-\gamma}\norm{\int_{a}^{t}\psi^{\prime}(s)(\Psi(t,s))^{\alpha-1}P_{\alpha}(\Psi(t,s))Bu^{n_k}(s)ds\right.\\
			&\left.-\int_{a}^{t}\psi^{\prime}(s)(\Psi(t,s))^{\alpha-1}P_{\alpha}(\Psi(t,s))Bu^0(s)ds}\to 0~\text{as}~k\to \infty, \forall t\in [a,b].
	\end{align*}
	In the above
	\begin{equation}
		q^*(t)=S_{\alpha,\beta}(\Psi(t,a))x_0+\int_{a}^{t}\psi^{\prime}(s)(\Psi(t,s))^{\alpha-1}P_{\alpha}(\Psi(t,s))Bu^0(s)ds, t\in [a,b],
	\end{equation}
	It is clear by the above expression the function $q^*(\cdot)\in C^{1-\gamma;\psi}([a,b],X)$ is the unique mild solution of the equation \eqref{4.1} with the control $u^0(\cdot)\in L^2([a,b], Y)$. Since the weak limit is unique, then by combining the convergences, we obtain that $q^*(t)=q^0(t), \forall t\in [a,b]$. Hence the function $q^0(\cdot)$ is the unique mild solution of the system \eqref{4.1} with the control $u^0(\cdot)\in L^2([a,b], Y)$ and also the whole sequence $q^n(\cdot)\to q^0(\cdot)\in C^{1-\gamma;\psi}([a,b], X)$. Consequently, we have $(q^0(\cdot),u^0(\cdot))\in A_{ad}$. 
	
	It remains to show that the functional $J(\cdot,\cdot)$ attains its minimum at $(q^0(\cdot),u^0(\cdot))$, that is, $J=J(q^0(\cdot),u^0(\cdot))$. Since the cost functional $J(\cdot,\cdot)$ given in  is convex and continuous on $C^{1-\gamma;\psi}([a,b],X)\times L^2([a,b],Y)$, it follows that $J(\cdot,\cdot)$ is sequentially weakly lower semicntinuous. That is, 
	\begin{equation}
		(q^n(\cdot),u^n(\cdot))\rightharpoonup (q^0(\cdot),u^0(\cdot))~\text{in}~C^{1-\gamma;\psi}([a,b],X)\times L^2([a,b],Y)~\text{as}~n\to \infty 
	\end{equation}
	$~\text{implies}~J(q^0(\cdot),u^0(\cdot))\le \liminf_{n\to \infty}J(q^n(\cdot),u^n(\cdot))$.
	Hence we obtain
	\begin{equation}
		J\le J(q^0(\cdot),u^0(\cdot))\le \liminf_{n\to \infty}J(q^n(\cdot),u^n(\cdot))=\lim_{n\to \infty} J(q^n(\cdot), u^n(\cdot))=J,
	\end{equation}
	and thus, $(q^0(\cdot),u^0(\cdot))$ is a minimizer of the problem \eqref{4.47}. Note that the cost functional defined in  is convex, the constraint given in \eqref{4.1} is linear, and the admissible class $L^2([a,b], Y)$ is convex, then the optimal control obtained above is unique. 
\end{proof}
\begin{lemma}
	If $u(\cdot)$ is the optimal control satisfying the problem \eqref{4.43}, then $u(\cdot)$ is given by
	\begin{equation}
		u(t)=\psi^{\prime}(t)(\Psi(b,t))^{\alpha-1}B^*P^*_{\alpha}(\Psi(b,t))J(\lambda I +R(b)J)^{-1}[x_b-S_{\alpha,\beta}(\Psi(b,a))x_0], t\in [a,b].
	\end{equation}
\end{lemma}
\begin{proof}
	Let us first consider the functional 
	\begin{equation}
		I(\epsilon)=J(q_{u+\epsilon w}(\cdot), u+\epsilon w(\cdot)),
	\end{equation}
	where $(q(\cdot),u(\cdot))$ is the optimal solution of \eqref{4.43} and $w(\cdot)\in L^2([a,b],Y)$. Also, the function $q_{u+\epsilon w}(\cdot)$ is the unique mild solution of \eqref{4.1} corresponding to the control $u+\epsilon w$. Then, it is immediate that
	\begin{equation}
		q_{u+\epsilon w}(t)=S_{\alpha,\beta}(\Psi(t,a))x_0+\int_{a}^{t}\psi^{\prime}(s)(\Psi(t,s))^{\alpha-1}P_{\alpha}(\Psi(t,s))B(u+\epsilon w)(s)ds, t\in [a,b],
	\end{equation}
	$\epsilon=0$ is the critical point of $I(\epsilon)$. We now evaluate the first variation of the cost functional $J$ defined in.

	Since $X$ is a separable reflexive Banach space with a strictly convex dual $X^*$,then the fact 8.12 in \cite{fabian2001functional} ensures that the norm $\norm{\cdot}_X$ is Gateaux differentiable. Moreover, the Gatauex derivative $\partial_x$ of the function $\phi(x)=\frac{1}{2}\norm{x}_X^2$ is the duality map, that is,
	\begin{equation}
		\langle \partial_x \phi(x),y\rangle=\frac{1}{2}\frac{d}{d\epsilon}\norm{x+\epsilon y}_X^2|_{\epsilon=0}=\langle J[x],y\rangle.
	\end{equation}
	Then we compute
	\begin{align*}
		&\frac{d}{d\epsilon}I(\epsilon)|_{\epsilon=0}\\=&\frac{d}{d\epsilon}\left[\norm{q_{u+\epsilon w}(b)-x_b}^2+\lambda\int_{a}^{b}\norm{u(t)+\epsilon w(t)}^2dt\right]_{\epsilon=0}\\
		=&2\left[\langle J(q_{u+\epsilon w}(b)-x_b), \frac{d}{d\epsilon}(q_{u+\epsilon w}-x_b)\rangle\right.\\
		&\left.\hspace{4cm}+2\lambda \int_{a}^{b}\langle u(t)+\epsilon w(t), \frac{d}{d\epsilon}(u(t)+\epsilon w(t))\rangle dt\right]_{\epsilon=0}\\
		=&2\langle J(q(b)-x_b), \int_{a}^{b}\psi^{\prime}(s)(\Psi(b,s))^{\alpha-1}P_{\alpha}(\Psi(b,s))Bu(s)ds\rangle+2\lambda \int_{a}^{b}\langle u(t), w(t)\rangle dt.
	\end{align*}
	By taking the first variation of the cost functional as zero, we deduce that
	\begin{equation}
		\langle J(q(b)-x_b), \int_{a}^{b}\psi^{\prime}(s)(\Psi(b,s))^{\alpha-1}P_{\alpha}(\Psi(b,s))Bw(s)ds\rangle+\lambda \int_{a}^{b}\langle u(t), w(t)\rangle dt=0.
	\end{equation}
	Consequently,
	\begin{equation}
		\int_{a}^{b}\psi^{\prime}(s)(\Psi(b,s))^{\alpha-1}\langle J(q(b)-x_b), P_{\alpha}(\Psi(b,s))Bw(s)\rangle ds+\lambda \int_{a}^{b}\langle u(t), w(t)\rangle dt=0.
	\end{equation}
	and hence
	\begin{equation}
		\int_{a}^{b}\langle\psi^{\prime}(s)(\Psi(b,s))^{\alpha-1} B^*P_{\alpha}(\Psi(b,s))^*J(q(b)-x_b)+\lambda u(s),w(s)\rangle ds=0. 
	\end{equation}
	Since $w\in L^2([a,b],Y)$ is arbitrary element, we can choose $w$ to be 
	\begin{equation}
		\psi^{\prime}(s)(\Psi(b,s))^{\alpha-1} B^*P_{\alpha}(\Psi(b,s))^*J(q(b)-x_b)+\lambda u(s), s\in [a,b].
	\end{equation}
	Then it follows that the optimal control is given by
	\begin{equation}
		u(t)=-\frac{1}{\lambda}\psi^{\prime}(t)(\psi(b)-\psi(t))^{\alpha-1}B^*P_{\alpha}^*(\psi(b)-\psi(t))J(q(b)-x_b), t\in [a,b].
	\end{equation}
	Using the above expression of the control, we find that
	\begin{align*}
		q(b)=&S_{\alpha,\beta}(\Psi(b,a))x_0\\
		&-\frac{1}{\lambda}\int_{a}^{b}\left\{\psi^{\prime}(s)(\Psi(b,s))^{\alpha-1}\right\}^2P_{\alpha}(\Psi(b,s))BB^*P_{\alpha}^*(\Psi(b,s))J(q(b)-x_b)ds\\
		=&S_{\alpha,\beta}(\Psi(b,a))x_0-\frac{1}{\lambda}R(b)J(q(b)-x_b),
	\end{align*}
	where
	\begin{equation}
		R(b)=\int_{a}^{b}\left\{\psi^{\prime}(s)(\Psi(b,s))^{\alpha-1}\right\}^2P_{\alpha}(\Psi(b,s))BB^*P_{\alpha}^*(\Psi(b,s))ds.
	\end{equation}
	From there, we get the expression of the optimal control as
	\begin{equation}
		u(t)=\psi^{\prime}(t)(\Psi(b,t))^{\alpha-1}B^*P^*_{\alpha}(\Psi(b,t))J(\lambda I +R(b)J)^{-1}[x_b-S_{\alpha,\beta}(\Psi(b,a))x_0], t\in [a,b].
	\end{equation}
\end{proof}
\section{Approximate Controllability}
We now derive the approximate controllability problems.
We consider the fractional evolution inclusion involving $\psi$-Hilfer fractional derivative
\begin{equation}\label{5.1}
	\begin{cases}
		^{H}D^{\alpha,\beta;\psi}_{a+} q(t)\in Aq(t)+F(t,q(t))+Bu(t), t\in [a,b]\\
		(I_{a+}^{1-\gamma;\psi}q)(a)=x_0.
	\end{cases}
\end{equation}
Here $A$ is a closed linear operator generating a strongly continuous semigroup $\{T(t)\}$ of bounded linear operators defined on a Banach space $X$. $^{H}D^{\alpha,\beta;\psi}_{a+}$ is the left $\psi$-Hilfer fractional derivative, $F:[a,b]\times X\multimap X$ is a given multivalued map and $B: Y\to X$ is a bounded linear map that describes the control action, $Y$ being a separable Hilbert space. 
\begin{theorem}\label{Thm3.1}
	Assume the hypotheses (T) and (C) hold. Also, suppose that the multivalued nonlinearity $F$ satisfies conditions (F1)-(F4). Then the system \eqref{5.1} is approximately controllable in $[a,b]$.
\end{theorem}

\begin{proof}
	For each $\epsilon>0$, we define a map $\Gamma_{\epsilon}:C^{1-\gamma;\psi}([a,b],X)\multimap C^{1-\gamma;\psi}([a,b],X)$, where for each $q\in C^{1-\gamma;\psi}([a,b],X)$, $\Gamma_{\epsilon}(q)$ consists functions $y\in C^{1-\gamma;\psi}([a,b],X)$ which satisfies
	\begin{equation}\label{5.2}				
		y(t)=S_{\alpha,\beta}(\Psi(t,a))x_0+\int_{a}^{t}\psi^{\prime}(s)(\Psi(t,s))^{\alpha-1}P_{\alpha}(\Psi(t,s))[f(s)+Bu(s)]ds,
	\end{equation}
	$ t\in [a,b],$ where $f\in S_F(q)$ and $u\in L^2([a,b],Y)$ is given by
	\begin{equation}\label{5.3}
		u(t)=\psi^{\prime}(t)(\Psi(b,t))^{\alpha-1}B^*P^*_{\alpha}(\Psi(b,t))J(\epsilon I +R(b)J)^{-1}N(f), t\in [a,b].
	\end{equation}
	In the above 
	\begin{equation}\label{5.4}
		N(f)=x_1-S_{\alpha,\beta}(\Psi(b,a))x_0-\int_{a}^{b}\psi^{\prime}(s)K_{\alpha}(\Psi(b,s))f(s)ds.
	\end{equation}
	For the sake of convenience, we distinguish the proof in several steps:\\

	\textbf{STEP-I:} $\Gamma_{\epsilon}$ maps bounded sets into bounded sets.
	Let $$Q_R^{\epsilon}=\{y\in C^{1-\gamma;\psi}([a,b],X): \norm{y}_{C^{1-\gamma;\psi}([a,b],X)}\le R \}.$$
	For any $y\in \Gamma_{\epsilon}(Q_R^{\epsilon})$, \eqref{5.2}-\eqref{5.4} hold. Then we compute
	\begin{align*}
		\norm{y(t)}\le &	\norm{S_{\alpha,\beta}(\Psi(t,a))x_0}+\int_{a}^{t}  \norm{K_{\alpha}(\Psi(t,s))  [f(s)+Bu(s)]\psi'(s)} ds, t\in (a,b].
	\end{align*}
	Using proposition \ref{Prop4.3} and Hypothesis (F4) we obtain for $t \in (a,b]$,
	\begin{equation}\label{5.5}
		\begin{aligned}
			&\norm{y(t)}\\\le & \norm{S_{\alpha,\beta}(\Psi(t,a))x_0}+\int_{a}^{t} \psi'(s)  (\Psi(t,s))^{\alpha-1}\frac{M}{\Gamma(\alpha)} \norm{f(s)+Bu(s)} ds\\
			\le &(\Psi(t,a))^{\gamma-1}\frac{M}{\Gamma(\gamma)}\norm{x_0}\\
			&+\frac{M}{\Gamma(\alpha)} \int_{a}^{t} \psi'(s)  (\Psi(t,s))^{\alpha-1}m(s) ds+\frac{M}{\Gamma(\alpha)} \int_{a}^{t} \psi'(s)  (\Psi(t,s))^{\alpha-1}\norm{Bu(s)} ds. 
		\end{aligned}
	\end{equation}	
	We also compute
	\begin{align*}
		&\norm{N(f)}\\
		= &
		\norm{	x_1-S_{\alpha,\beta}(\Psi(b,a))x_0-\int_{a}^{b}\psi^{\prime}(s)K_{\alpha}(\Psi(b,s))f(s)ds}\\
		\le& \norm{x_1}+(\Psi(b,a))^{\gamma-1}\frac{M}{\Gamma(\gamma)}\norm{x_0}+\frac{M}{\Gamma(\alpha)}\int_{a}^{b}\psi^{\prime}(s)(\Psi(b,s))^{\alpha-1}m(s)ds=K_0~\text{say}.
	\end{align*}
	Using the expression \eqref{5.3} we estimate 
	\begin{align*}
		&	\norm{\int_{a}^{t}\psi^{\prime}(s)K_{\alpha}(\Psi(t,s))Bu(s)ds}\\		=&\int_{a}^{t}\psi^{\prime}(s)(\Psi(t,s))^{\alpha-1}\norm{P_{\alpha}(\Psi(t,s))Bu(s)}ds\\
		=&\int_{a}^{t}[\psi^{\prime}(s)]^2(\Psi(t,s))^{\alpha-1}(\Psi(b,s))^{\alpha-1}
		\\
		&\hspace{3cm}\times\norm{P_{\alpha}(\Psi(t,s))BB^*P^*_{\alpha}(\Psi(b,s))J(\epsilon I +R(b)J)^{-1}N(f)}ds\\
		\le& \tilde{K}\int_{a}^{t}\psi^{\prime}(s)(\Psi(t,s))^{\alpha-1}(\Psi(b,s))^{\alpha-1}\\
		&\hspace{3cm}\times\norm{P_{\alpha}(\Psi(t,s))BB^*P^*_{\alpha}(\Psi(b,s))J(\epsilon I +R(b)J)^{-1}N(f)}ds.
	\end{align*}
	In the above, $\tilde{K}=\sup_{t \in [a,b]}\abs{\psi^{\prime}(t)}$. 	Again, by means of Hypothesis (F4) we can write
	\begin{align*}
		&	\norm{\int_{a}^{t}\psi^{\prime}(s)K_{\alpha}(\Psi(t,s))Bu(s)ds}\\
		\le& \tilde{K}\frac{M}{\Gamma(\alpha)}\norm{B}\norm{B^*}\frac{M}{\Gamma(\alpha)}\int_{a}^{t}\psi^{\prime}(s)(\Psi(t,s))^{\alpha-1}(\Psi(b,s))^{\alpha-1}\\
		&\hspace{7cm}\times\norm{J(\epsilon I +R(b)J)^{-1}N(f)}ds\\
		\le&\frac{1}{\epsilon} \tilde{K}\frac{M}{\Gamma(\alpha)}\norm{B}\norm{B^*}\frac{M}{\Gamma(\alpha)}\int_{a}^{t}\psi^{\prime}(s)(\Psi(t,s))^{\alpha-1}(\Psi(b,s))^{\alpha-1}\\
		&	\hspace{7cm}\times\norm{\epsilon(\epsilon I +R(b)J)^{-1}N(f)}ds.	
	\end{align*}
	By virtue of Lemma 2.2 of \cite{mahmudov2003approximate} we obtain
	\begin{align*}
		&	\norm{\int_{a}^{t}\psi^{\prime}(s)K_{\alpha}(\Psi(t,s))Bu(s)ds}\\	
		\le&\frac{1}{\epsilon} \tilde{K}\frac{M}{\Gamma(\alpha)}\norm{B}\norm{B^*}\frac{M}{\Gamma(\alpha)}\int_{a}^{t}\psi^{\prime}(s)(\Psi(t,s))^{\alpha-1}(\Psi(b,s))^{\alpha-1}\norm{N(f)}ds.
	\end{align*}
	Due to the nonincreasing nature of the map $t\to t^{\alpha-1}, \alpha-1<0$ we obtain
	\begin{align}
		&	\norm{\int_{a}^{t}\psi^{\prime}(s)K_{\alpha}(\Psi(t,s))Bu(s)ds}\\	
		\le&\frac{1}{\epsilon} \tilde{K}\frac{M}{\Gamma(\alpha)}\norm{B}\norm{B^*}\frac{M}{\Gamma(\alpha)}\norm{N(f)}\int_{a}^{t}\psi^{\prime}(s)(\Psi(t,s))^{\alpha-1}(\Psi(t,s))^{\alpha-1}ds.
	\end{align}
	Putting these pieces together, we obtain from \eqref{5.5}
	\begin{align*}
		\norm{y(t)}\le& \frac{M}{\Gamma(\gamma)}(\Psi(t,a))^{\gamma-1}\norm{x_0}+\frac{M}{\Gamma(\alpha)} \int_{a}^{t} \psi'(s)  (\Psi(t,s))^{\alpha-1}m(s) ds\\
		&+\frac{1}{\epsilon} \tilde{K}\frac{M}{\Gamma(\alpha)}\norm{B}\norm{B^*}\frac{M}{\Gamma(\alpha)}\norm{N(f)}\int_{a}^{t}\psi^{\prime}(s)(\Psi(t,s))^{\alpha-1}(\Psi(t,s))^{\alpha-1}ds.
	\end{align*}
	Hence
	\begin{align*}
		&\norm{(\Psi(t,a))^{1-\gamma}y(t)}\\
		\le& \frac{M}{\Gamma(\gamma)}\norm{x_0}+(\Psi(t,a))^{1-\gamma}\frac{M}{\Gamma(\alpha)} \int_{a}^{t} \psi'(s)  (\Psi(t,s))^{\alpha-1}m(s) ds\\
		&\hspace{2cm}+(\Psi(t,a))^{1-\gamma}\frac{1}{\epsilon} \tilde{K}\frac{M}{\Gamma(\alpha)}\norm{B}\norm{B^*}\frac{M}{\Gamma(\alpha)}\norm{N(f)}\frac{(\Psi(t,a))^{2\alpha-1}}{2\alpha-1}\\
		\le& \frac{M}{\Gamma(\gamma)}\norm{x_0}+(\Psi(t,a))^{1-\gamma}\frac{M}{\Gamma(\alpha)} \int_{a}^{t} \psi'(s)  (\Psi(t,s))^{\alpha-1}m(s) ds\\
		&\hspace{2cm}+(\Psi(b,a))^{1-\gamma}\frac{1}{\epsilon} \tilde{K}\frac{M}{\Gamma(\alpha)}\norm{B}\norm{B^*}\frac{M}{\Gamma(\alpha)}\norm{N(f)}\frac{(\Psi(b,a))^{2\alpha-1}}{2\alpha-1}.
	\end{align*}
	Due to Hypothesis (F4) we guarantee that there exists $r>0$ such that
	\begin{equation}\label{M3}
		\frac{M}{\Gamma(\gamma)}\norm{x_0}+\frac{M}{\Gamma(\alpha)}\sup_{t \in [a,b]}\left[(\Psi(t,a))^{1-\gamma}\int_{a}^{t}\psi^{\prime}(s)(\Psi(t,s))^{\alpha-1}m(s)ds\right]\le r.
	\end{equation}
	Therefore, taking supremum over $[a,b]$ and noting \eqref{M3} we see that
	\begin{align*}
		\norm{y}_{C^{1-\gamma;\psi}([a,b],X)}\le M_0~\text{for some}~M_0>0~\text{for all}~y\in B_R.
	\end{align*}
	Hence we see that
	\begin{equation}
		\norm{\Gamma_{\epsilon}(q)}_{C^{1-\gamma;\psi}([a,b],X)}\le M_0.
	\end{equation} 
	This completes STEP-I.\\
	
	\textbf{STEP-II} The image of the set
	\begin{equation*}
		Q_R^{\epsilon}=\{y\in C^{1-\gamma;\psi}([a,b],X): \norm{y}_{C^{1-\gamma;\psi}([a,b],X)}\le R\}
	\end{equation*} 
	under the map $\Gamma_{\epsilon}$ is equicontinuous for every $R>0$. \\
	Proceeding as in the proof of equicontinuity part of Lemma \ref{Lem 5.8}, the equicontinuity of the set $\{\Gamma_{\epsilon}(Q_R^{\epsilon})\}$ can be proved.\\

	\textbf{STEP-III} Fix $\epsilon>0$. We show that the multivalued map $\Gamma_{\epsilon}$ maps some nonempty compact convex set $Q^{\epsilon}$ into itself. In STEP-I, we see $\Gamma_{\epsilon}(Q_R^{\epsilon})\subset Q_R^{\epsilon}$. Also mimicking the same proof of Lemma \ref{Lem 5.8} we can easily prove  $\Gamma_{\epsilon}(Q_R^{\epsilon})$ is relatively compact in $C^{1-\gamma;\psi}([a,b],X)$. 	
	Define $	Q^{\epsilon}=\overline{conv}\{\Gamma_{\epsilon}(Q_R^{\epsilon})\}, $
	then it follows that $Q^{\epsilon}$ is a nonempty compact convex subset of $C^{1-\gamma;\psi}([a,b], X)$ and 
	\begin{equation*}
		Q^{\epsilon}=\overline{conv}\{\Gamma_{\epsilon}(Q_R^{\epsilon})\}\subset \overline{conv}Q_R^{\epsilon}=Q_R^{\epsilon}.
	\end{equation*}
	Therefore, $\Gamma_{\epsilon}(Q^{\epsilon})\subset Q^{\epsilon}.$
	
	\textbf{STEP-IV}  In this part we prove the existence of fixed points of the multivalued map $\Gamma_{\epsilon}$ for each $\epsilon>0.$
	Let $W_{\epsilon}: Q^{\epsilon}\multimap L^1([a,b],X)$ be defined as follows:\\
	for each $q\in Q^{\epsilon}, W_{\epsilon}(q)$ consists of all those functions $f:[a,b]\to X$ such that $f\in L^1([a,b],X):f(t)\in F(t,q(t)) ~\text{a.a.}~t\in [a,b]$. $W_{\epsilon}(q)$ is nonempty by Hypothesis (F1) and Aumann selection theorem \cite{wagner1977survey}. Also, $W_{\epsilon}$ is lower semicontinuous \cite{li2003controllability}
	and has closed and convex values. Thus by Michael's Selection theorem \cite{michael2003continuous}, there is a continuous function $\Lambda_{\epsilon}: Q^{\epsilon}\to L^1([a,b], X)$ such that $\Lambda_{\epsilon}(q)\in W_{\epsilon}(q)$ for every $q\in Q^{\epsilon}$. Therefore, $S_{\epsilon}\circ\Lambda_{\epsilon}:Q^{\epsilon}\to Q^{\epsilon}$ is a single-valued mapping and $S_{\epsilon}\circ\Lambda_{\epsilon}(q)\in \Gamma_{\epsilon}(q)$ for every $q\in Q^{\epsilon}$, where $S_{\epsilon}: L^1([a,b],X)\to C^{1-\gamma;\psi}([a,b],X)$ is defined as follows: if $f\in L^1([a,b],X)$ then $S_{\epsilon}(f)$ satisfy. Moreover, the map $S_{\epsilon}\circ\Lambda_{\epsilon}: Q^{\epsilon}\to Q^{\epsilon}$ is a continuous map. \\
	Therefore by Schauder's fixed point theorem, for each $\epsilon>0$, we obtain a solution $q_{\epsilon}$ of our problem. Hence we obtain
	\begin{equation}
		q_{\epsilon}(t)=S_{\alpha,\beta}(\Psi(t,a))x_0+\int_{a}^{t}\psi^{\prime}(s)(\Psi(t,s))^{\alpha-1}P_{\alpha}(\Psi(t,s))[f_{\epsilon}(s)+Bu_{\epsilon}(s)]ds, t\in (a,b],
	\end{equation}
	where $f_{\epsilon}\in S_F(q_{\epsilon})$ and $u_{\epsilon}\in L^2([a,b],Y)$ is given by
	\begin{equation}
		u_{\epsilon}(t)=\psi^{\prime}(t)(\Psi(b,t))^{\alpha-1}B^*P^*_{\alpha}(\Psi(b,t))J(\epsilon I +R(b)J)^{-1}N(f_{\epsilon}), t\in [a,b].
	\end{equation}
	In the above 
	\begin{equation}
		N(f_{\epsilon})=x_1-S_{\alpha,\beta}(\Psi(b,a))x_0-\int_{a}^{b}\psi^{\prime}(s)K_{\alpha}(\Psi(b,s))f_{\epsilon}(s)ds.
	\end{equation}

	\textbf{STEP-V} In this part we conclude this section by proving the approximate controllability result.

	Let $q_{\epsilon}\in Q^{\epsilon}$ be the fixed points of the map $\Gamma_{\epsilon}$ for each $\epsilon>0$. Clearly, $q_{\epsilon}$ is a solution of the problem \eqref{5.1} for each $\epsilon>0$. We finish the proof by showing that $q_{\epsilon}(b)\to x_1$. In order to do this we need the convergence of the sequence $\{f_{\epsilon}\}\subset L^1([a,b],X)$.

	By hypothesis (F4), we obtain $\{f_{\epsilon}\}$ is integrably bounded and $\{f_{\epsilon}(t)\}$ is weakly relatively compact set in $X$. Therefore Dunford Pettis theorem \ref{DF} guarantees that the set $\{f_{\epsilon}\}\subset L^1([a,b],X)$ is relatively weakly compact. Hence we obtain $f_{\epsilon}\rightharpoonup f$ in $L^1([a,b],X)$ upto a subsequence. Therefore using Corollary \ref{Corollary}, we obtain
	\begin{align*}
		N(f_{\epsilon})&\to N(f)=x_1-S_{\alpha,\beta}(\Psi(b,a))x_0-\int_{a}^{b}\psi^{\prime}(s)K_{\alpha}(\Psi(b,s))f(s)ds.
	\end{align*}
	We now estimate
	\begin{align*}
		q_{\epsilon}(b)=&S_{\alpha,\beta}(\Psi(b,a)x_0+\int_{a}^{b}\psi^{\prime}(s)(\Psi(b,s))^{\alpha-1}P_{\alpha}(\Psi(b,s))[f_{\epsilon}(s)+Bu_{\epsilon}(s)]ds\\
		=&x_1-N(f_{\epsilon})+R(b)J(\epsilon I+R(b)J)^{-1}N(f_{\epsilon})\\
		=&x_1-N(f_{\epsilon})+N(f_{\epsilon})-\epsilon (\epsilon I+R(b)J)^{-1}N(f_{\epsilon})\\
		=&x_1-\epsilon (\epsilon I+R(b)J)^{-1}N(f_{\epsilon}).
	\end{align*}
	Therefore, the fact that $N(f_{\epsilon})\to N(f)$ in $X$ and with Lemma 4.4 of \cite{pinaud2020controllability} we obtain $q_{\epsilon}(b)\to x_1$ as desired.
\end{proof}

\section{Approximate Controllability Application}
Let us consider the following diffusion control system
\begin{equation}\label{M1}
	\begin{cases}
		^HD^{\alpha,\beta;\psi}_{0+} y(t,\xi)\in\frac{\partial^2 y(t,\xi)}{\partial \xi^2}+\displaystyle a(t,\xi)W\left(\xi,\int_{\Omega}\phi(\xi)y(t,\xi)d\xi\right)+Bv(t,\xi),\\
		t\in J=[0,T], ~\text{a.e.}~\xi\in \Omega=[0,\pi]\\
		y(t,0)=0, y(t,\pi)=0, t\in J\\
		I_{0+}^{1-\gamma;\psi}y(t,\xi)=x_0(\xi), t\in J, \xi\in \Omega.
	\end{cases}
\end{equation}
Here $	^HD^{\alpha,\beta;\psi}_{0+}$ is the $\psi$-Hilfer fractional derivative of order $\alpha$ and type $\beta$, $W:\Omega\times \mathbb{R}\to \mathbb{R}$ is a function with $W(\cdot,r)\in L^2(\Omega)\cap ~AC_{loc}(\Omega)$ for every $r\in \mathbb{R}$ and satisfying
\begin{equation}\label{W1}
	\abs{\frac{\partial W(\xi,r)}{\partial \xi}}\le l(\xi)~\text{for a.e.}~\xi\in \Omega, ~\text{for every}~r\in \mathbb{R},
\end{equation}
where $l\in L^1(\mathbb{R})$ and
\begin{equation}\label{W2}
	f_1(\xi,r)\le W(\xi,r)\le f_2(\xi,r)~\text{for every}~\xi\in \Omega, ~\text{for every}~r\in \mathbb{R},
\end{equation}
with $f_1, f_2:\Omega\times \mathbb{R}\to \mathbb{R}$ are functions such that for $i=1,2$,
\begin{equation}\label{W3}
	f_i(\cdot,r)\in L^2(\Omega)\cap ~AC_{loc}(\Omega)~\text{for every}~r\in \mathbb{R},
\end{equation}
and
\begin{equation}\label{W4}
	\abs{\frac{\partial f_i(\xi,r)}{\partial \xi}}\le l(\xi)~\text{for a.e.}~\xi\in \Omega, ~\text{for every}~r\in \mathbb{R};
\end{equation}
moreover, there exists $K_1>0$ such that
\begin{equation}\label{W5}
	\abs{f_i(\xi,r)}\le K_1, \xi\in \Omega, r\in \mathbb{R};
\end{equation} 
further
\begin{equation}\label{W6}
	f_1(\xi,r)\le f_2(\xi,r)~\text{for every}~\xi\in \Omega, r\in \mathbb{R},
\end{equation}
and
\begin{equation}\label{W7}
	f_1(\xi,r_0)\ge \limsup_{r\to r_0}f_1(\xi,r)~\text{and}~	f_2(\xi,r_0)\le \liminf_{r\to r_0}f_2(\xi,r),
\end{equation}
for every $\xi\in \Omega, r_0\in \mathbb{R}$.
In the above $AC_{loc}(\Omega)$ denotes the space of all the locally absolutely continuous functions 
defined on $\Omega$.
The prototype for the functions $f_1, f_2$ and $\phi$ appearing can be found in \cite{malaguti2016nonsmooth}.
Also, the function $a(t,\xi):[0,T]\times \Omega\to \mathbb{R}$ is such that
\begin{itemize}
	\item[(a1)] $(t,\xi)\mapsto a(t,\xi)$ is measurable for all $(t,\xi)\in [0,T]\times \Omega$.
	\item[(a2)] $\xi\mapsto a(t,\xi)$ is continuous.
	\item[(a3)] There exists $m\in L^1([0,T])$ such that 
	\begin{equation}
		\abs{a(t,\xi)}\le m(t)~\text{and}~\lim_{t\to 0+}(\psi(t)-\psi(0))^{1-\gamma}I_{0+}^{\alpha;\psi}m(t)=0.
	\end{equation}
\end{itemize}
From the properties assigned on $W$, we see that jump discontinuities appear in the connection between the integral $\displaystyle\int_{\Omega}\phi(\xi)y(t,\xi)d\xi$ and
the control function $W(\xi, r)$. Therefore, differential inclusion comes into the picture.

Our approach to the problem is to rewrite the control process as an abstract problem driven by a fractional inclusion in the space $X=L^2(\Omega)$. To this aim, we put:
\begin{equation*}
	u(t)(\xi)=v(t,\xi) ~\text{and}~q(t)(\xi)=y(t,\xi), t\in J, \xi\in \Omega.
\end{equation*}
Clearly, $q:[0,T]\to L^2(\Omega,\mathbb{R})$.\\ 
We define the opertaor $A: D(A)\subset X\to X$ is defined by $Ay=y^{\prime \prime}$, where the domain $D(A)$ is given by
\begin{equation}
	D(A)=\{y\in X: y, y^{\prime}~\text{are absolutely continuous,}~ y^{\prime \prime }\in X, y(0)=y(\pi)=0\}. 
\end{equation}
Then $A$ can be written as 
\begin{equation}
	Ay=-\sum_{n=1}^{\infty}n^2\langle y,e_n\rangle e_n, y\in D(A),
\end{equation}
where $e_n(\xi)=\sqrt{\frac{2}{\pi}}\sin(n \xi), n=1,2,3,\cdot\cdot\cdot$ is an orthonormal basis of $X$. It is well known that $A$ is the infinitesimal generator of a differentiable semigroup $T(t), t>0$ in $X$ given by
\begin{equation}
	T(t)y=\sum_{n=1}^{\infty}e^{-n^2t}\langle y, e_n\rangle e_n,
\end{equation}
and \begin{equation}
	\norm{T(t)}\le e^{-1}<1=M.
\end{equation}
%
%
We define $F:[0,T]\times L^2(\Omega)\multimap L^2(\Omega)$ as
\begin{equation*}
	F(t,x)(\xi)=a(t,\xi)H(x), (t,x)\in [0,T]\times L^2(\Omega),
\end{equation*}
where the multivalued map $H: L^2(\Omega)\multimap L^2(\Omega)$ is defined as
\begin{align*}
	H(x)=&\left\{y\in L^2(\Omega)\cap ~AC_{loc}(\Omega): \abs{y^{\prime}(\xi)}\le l(\xi)~\text{for a.a.}~ \xi\in \Omega~\text{and}~\right.\\
	&\left. f_1\left(\xi,\int_{\Omega}\phi(\omega)x(\omega)d\omega \right)\le y(\xi)\le f_2\left(\xi,\int_{\Omega}\phi(\omega)x(\omega)d\omega\right)
	\right\}.
\end{align*}
By the conditions \eqref{W1}-\eqref{W7}, it is immediate that the multimap $H$ is well defined.\\
Take $Y=L^2(\Omega)$ and $B: Y\to X $ to be the identity operator.
Now, we can write the reformulation of equation \eqref{M1} as the following semilinear evolution inclusion in the Banach space $X=L^2(\Omega)$:
\begin{equation*}
	\begin{cases}
		^HD_{0+}^{\alpha,\beta;\psi}q(t)\in Aq(t)+F(t,q(t))+Bu(t), t\in [0,T]\\
		I_{0+}^{1-\gamma;\psi}q(0)=x_0.
	\end{cases}
\end{equation*}
Of course, the solutions to these equations give rise to solutions for \eqref{M1}.\\
We now verify the assumptions required to prove the approximate controllability results.

Consider the following linear initial boundary value problem of fractional parabolic control system with $\psi$-Hilfer fractional derivative
\begin{equation}\label{6.14}
	\begin{cases}
		^HD^{\alpha,\beta;\psi}_{0+} q(t)=Aq(t)+Bu(t), t\in J=[0,T], \xi\in \Omega=[0,\pi];\\
		q(0)=q(\pi)=0, t\in J=[0,T]\\
		I_{0+}^{1-\gamma;\psi}q(0)=x_0, t\in [0,T].
	\end{cases}
\end{equation}
We show that the linear control system \eqref{6.14} is approximately controllable in $[0, T]$. Since $D(A)$ is dense in $X$, it is sufficient to show that for any $\eta\in D(A)$, there exists a control $u(\cdot)\in L^2(J,Y)$ such that $q(T)=\eta$, where $q$ is the solution of \eqref{6.14} corresponding to the control $u(\cdot)$. Define the operator $L: L^2(J,X)\to X$ by
\begin{equation}
	L(h)=\int_{0}^{T}\psi^{\prime}(s)(\Psi(t,s))^{\alpha-1}P_{\alpha}(\Psi(t,s))h(s)ds, h\in L^2(J,X).
\end{equation} 
We show for $\xi=\eta-S_{\alpha,\beta}(\psi(T)-\psi(0))x_0\in X$, there exists $\rho\in L^2(J,X)$ such that $L(\rho)=\eta-S_{\alpha,\beta}(\psi(T)-\psi(0))x_0$.\\
From this, it is immediate that the linear control system \eqref{6.14} is approximately controllable.
Choose $\rho\in L^2(J,X)$ as follows:
\begin{align*}
	\rho(t)=&\frac{[\Gamma(\alpha)]^2}{\psi(T)-\psi(0)}(\psi(T)-\psi(t))^{1-\alpha
	}\left[P_{\alpha}(\psi(T)-\psi(t))\xi\right.\\
	&\left.-2(\psi(t)-\psi(0))\frac{d}{dt}P_{\alpha}(\psi(T)-\psi(t))\xi\right].
\end{align*}
We now compute
\begin{align*}
	L(\rho)=&\int_{0}^{T}\psi^{\prime}(s)(\psi(T)-\psi(s))^{\alpha-1}P_{\alpha}(\psi(T)-\psi(s))\frac{[\Gamma(\alpha)]^2}{\psi(T)-\psi(0)}(\psi(T)-\psi(s))^{1-\alpha
	}\\
	&\hspace{2cm}\times\left[P_{\alpha}(\psi(T)-\psi(s))\xi-2(\psi(s)-\psi(0))\frac{d}{dt}P_{\alpha}(\psi(T)-\psi(s))\xi\right]ds\\
	=&\frac{[\Gamma(\alpha)]^2}{\psi(T)-\psi(0)}\int_{0}^{T}\psi^{\prime}(s)(\psi(T)-\psi(s))^{\alpha-1}P_{\alpha}(\psi(T)-\psi(s))\xi(\psi(T)-\psi(s))^{1-\alpha
	}\\
	&\hspace{2cm}\times\left[P_{\alpha}(\psi(T)-\psi(s))\xi-2(\psi(s)-\psi(0))\frac{d}{dt}P_{\alpha}(\psi(T)-\psi(s))\xi\right]ds\\
	=&\frac{[\Gamma(\alpha)]^2}{\psi(T)-\psi(0)}\int_{0}^{T}\psi^{\prime}(s)P_{\alpha}(\Psi(T,s))\xi\left[P_{\alpha}(\Psi(T,s))\xi\right.\\
	&\left.-2(\psi(s)-\psi(0))\frac{d}{dt}P_{\alpha}(\Psi(T,s))\xi\right]ds\\
	=&\frac{[\Gamma(\alpha)]^2}{\psi(T)-\psi(0)}\left[\int_{0}^{T}\psi^{\prime}(s)P^2_{\alpha}(\Psi(T,s))\xi ds\right.\\
	&\left.-2\int_{0}^{T}\psi^{\prime}(s)P_{\alpha}(\Psi(T,s))\xi(\psi(s)-\psi(0))\frac{d}{ds}P_{\alpha}(\Psi(T,s))\xi ds\right].
\end{align*}
Noting that
\begin{equation}
	\frac{d}{ds}(P_{\alpha}(\Psi(T,s)))^2=-2\psi^{\prime}(s)P_{\alpha}(\Psi(T,s))\frac{d}{ds}P_{\alpha}(\Psi(T,s)), 
\end{equation}
we obtain
\begin{align*}
	L(\rho)=&\frac{[\Gamma(\alpha)]^2}{\psi(T)-\psi(0)}\left[\int_{0}^{T}\psi^{\prime}(s)P^2_{\alpha}(\Psi(T,s))\xi ds\right.\\
	&\left.+\int_{0}^{T}(\psi(s)-\psi(0))\frac{d}{dt}[P_{\alpha}(\Psi(T,s))\xi]^2ds\right].
\end{align*}
Integrating by parts we obtain
\begin{align*}
	L(\rho)=&\frac{[\Gamma(\alpha)]^2}{\psi(T)-\psi(0)}\left[\int_{0}^{T}\psi^{\prime}(s)P^2_{\alpha}(\Psi(T,s))\xi ds-(\psi(T)-\psi(0))[P_{\alpha}(0)]^2\right.\\
	&\left.+\int_{0}^{T}\psi^{\prime}(s)[P_{\alpha}(\Psi(T,s))\xi]^2ds\right]=\xi.
\end{align*}
Therefore, we obtain
\begin{equation}
	S_{\alpha,\beta}(\psi(T)-\psi(0))x_0+\int_{0}^{T}\psi^{\prime}(s)(\Psi(T,s))^{\alpha-1}P_{\alpha}(\Psi(T,s))\rho(s) ds=\eta.
\end{equation}
We can choose $u=\rho\in L^2(J,U)$. As the operator $B$ is the identity operator on $X$, we can write
\begin{equation}
	S_{\alpha,\beta}(\psi(T)-\psi(0))x_0+\int_{0}^{T}\psi^{\prime}(s)(\Psi(T,s))^{\alpha-1}P_{\alpha}(\Psi(T,s))Bu(s) ds=\eta.
\end{equation}
This shows that the linear system is controllable in $D(A)$. Since $D(A)$ is dense in $X$, the linear system \eqref{6.14} is approximately controllable.

We now verify the multimap $F$ satisfies the Hypothesis (F1)-(F4). Hypothesis (F1) and (F2) are immediate from the following proposition.

\begin{proposition}
	For every $x\in L^2(\Omega)$, the set $H(x)$ is nonempty, closed and convex in $L^2(\Omega)$.
\end{proposition}
The following proposition verifies Hypothesis (F3).
\begin{proposition}
	The multimap $H$ is lower semicontinuous.
\end{proposition}
\begin{proof}
	Consider arbitrary $x\in L^2(\Omega), y\in H(x)$ and $\{x_n\}\subset L^2(\Omega)$ with $x_n\to x$ in $L^2(\Omega)$. If we can find $\{y_n\}\subset L^2(\Omega)$ with $y_n\in H(x_n)$ for all $n\in \mathbb{N}$ and such that $y_n\to y$ in $L^2(\Omega)$, then the lower semicontinuity of $H$ is proved.
	
	Putting 
	\begin{equation}
		r_n=\int_{\Omega}\phi(\xi)x_n(\xi)d\xi, r=\int_{\Omega}\phi(\xi)x(\xi)d\xi, n\in \mathbb{N},
	\end{equation}
	we have $r_n\to r$. For every $n\in \mathbb{N}$, let us define 
	\begin{equation}
		y_n(\xi)=\min\{p_n(\xi),f_2(\xi,r_n) \}~\text{for every}~\xi\in \Omega,
	\end{equation}
	where
	\begin{equation}
		p_n(\xi)=\max\{y(\xi), f_1(\xi,r_n)\}~\text{for every}~\xi\in \Omega.
	\end{equation}
	First of all, $y_n\in H(x_n)$. In fact, functions $y, f_1(\cdot,r_n), f_2(\cdot,r_n)$ belong to $AC_{loc}(\Omega)$, so $y_n\in ~AC_{loc}(\Omega)$ too. Moreover,
	\begin{equation}
		p_n^{\prime}(\xi)\in \left\{y^{\prime}(\xi), \frac{\partial f_1(\xi,r_n)}{\partial \xi}\right\}, ~\text{for every}~\xi\in \Omega\setminus N_1,
	\end{equation}
	and
	\begin{equation}
		y_n^{\prime}(\xi)\in \left\{p_n^{\prime}(\xi), \frac{\partial f_2(\xi,r_n)}{\partial \xi}\right\}, ~\text{for every}~\xi\in \Omega\setminus N_2,
	\end{equation}
	where $\mu(N_1)=\mu(N_2)=0$. So, we obtain
	\begin{equation}
		y_n^{\prime}(\xi)\in \left\{y^{\prime}(\xi), \frac{\partial f_1(\xi,r_n)}{\partial \xi}, \frac{\partial f_2(\xi,r_n)}{\partial \xi}\right\}~\text{for every}~\xi\in \Omega\setminus (N_1\cup N_2). 
	\end{equation}
	Therefore, by the definition of $H$, we obtain
	\begin{equation}
		\abs{y_n^{\prime}(\xi)}\le l(\xi)~\text{for a.a.}~\xi\in \Omega.
	\end{equation}
	The other conditions for $y_n$ belonging to $H(x_n)$ are easily satisfied.
	
	Now, we show that the sequence $\{y_n\}$ pointwise converges to $y$ in $\mathbb{R}$. Let us fix $\xi\in \Omega$. By \eqref{W7}, we have the following chain inequality
	\begin{equation}
		\Lambda(\xi)=\limsup_{n\to \infty} f_1(\xi,r_n)\le f_1(\xi,r)\le y(x)\le f_2(\xi,r)\le \liminf_{n\to \infty} f_2(\xi,r_n)=\lambda(\xi).
	\end{equation}
	If $\Lambda(\xi)<y(\xi)<\lambda(\xi)$, then $f_1(\xi,r_n)<y(\xi)<f_2(\xi,r_n)$ for $n$ large enough, so $y_n(\xi)=y(\xi)$ for $n>>1$ and the convergence is trivial.
	
	If $\Lambda(\xi)<y(\xi)=\lambda(\xi)$, there exists $n_0\in \mathbb{N}$ such that for $n\ge n_0$, we obtain $f_1(\xi,r_n)<y(\xi)$, then $p_n(\xi)=y(\xi)$, so
	\begin{equation}
		y_n(\xi)=
		\begin{cases}
			y(\xi), ~\text{if}~y(\xi)\le f_2(\xi,r_n)\\
			f_2(\xi,r_n), ~\text{if}~\lambda(\xi)>f_2(\xi,r_n).
		\end{cases}
	\end{equation}
	In this case, the sequence $\{y_n(\xi)\}_{n\ge n_0}$ splits almost into two subsequences; the first one is given by $y_{n_q}(\xi)=y(\xi)$ for every $q\in \mathbb{N}$, which trivially converges to $y(\xi)$; the second by $y_{n_p}(\xi)=f_2(\xi,r_n)$ for every $p\in \mathbb{N}$. For this second subsequences, on the one hand, we have
	\begin{equation}
		\liminf_{n\to \infty} y_{n_p}(\xi)=\liminf_{n\to \infty} f_2(\xi,r_n)\ge \liminf_{n\to \infty} f_2(\xi,r_n)=\lambda(\xi)=y(\xi).
	\end{equation}
	On the other, $y_{n_p}(\xi)=f_2(\xi, r_{n_p})<y(\xi)$, for every $p\in \mathbb{N}$, hence $\limsup_{n\to \infty} y_{n_p}(\xi)\le y(\xi)$. Therefore, $\lim_{n\to \infty} y_{n_p}(\xi)=y(\xi)$. If $\Lambda(\xi)=y(\xi)<\lambda(\xi)$, we reach the same conclusion with symmetric reasoning.
	
	If $\Lambda(\xi)=\lambda(\xi)$, then necessarily $\Lambda(\xi)=f_1(\xi,r)=f_2(\xi,r)=\lambda(\xi)$, implying $\Lambda(\xi)=y(\xi)=\lambda(\xi)$. Then
	\begin{equation}
		y_n(\xi)=\begin{cases}
			y(\xi), ~\text{if}~f_1(\xi,r_n)\le y(\xi)\le f_2(\xi,r_n)\\
			f_2(\xi,r_n)~\text{if}~y(\xi)>f_2(\xi,r_n)\\
			f_1(\xi,r_n)~\text{if}~y(\xi)<f_1(\xi,r_n),
		\end{cases}
	\end{equation}
	for every $n\in \mathbb{N}$. This time the subsequence $\{y_n(\xi)\}$ splits into almost three subsequences. By proceeding as in the previous two situations, it is easy to prove that each converges to $y(\xi)$. In conclusion, we have that $\lim_{n\to \infty}y_n(\xi)=y(\xi)$ for every $\xi\in \Omega$. Moreover, the set $\{y_n\}$ is bounded. This leads us to say that the pointwise convergence of $\{y_n\}$ is dominated. Therefore $y_n\to y$ in $L^2(\Omega)$. Consequently, $H$ is lower semicontinuous.
\end{proof}
We now verify Hypothesis (F4). Let $y\in H(x)$. The definition of the multimap $H$ permits
\begin{equation}
	f_1\left(\xi,\int_{\Omega}\phi(\omega)x(\omega)d\omega \right)\le y(\xi)\le f_2\left(\xi,\int_{\Omega}\phi(\omega)x(\omega)d\omega\right), \xi\in \Omega.
\end{equation}
Utilizing \eqref{W5} we obtain
\begin{equation}
	\abs{y(\xi)}\le \abs{f_1\left(\xi,\int_{\Omega}\phi(\omega)x(\omega)d\omega \right)}+\abs{f_2\left(\xi,\int_{\Omega}\phi(\omega)x(\omega)d\omega \right)}\le 2K_1, \xi\in \Omega.
\end{equation}
Therefore,
\begin{equation}
	\int_{\Omega}\abs{a(t,\xi)y(\xi)}^2d\xi\le 4K_1^2(m(t))^2,
\end{equation}
implying that
\begin{equation}
	\norm{F(t,x)}\le 2K_1m(t), t\in [0,T], x\in X.
\end{equation}
Under assumption (a3), it is obvious that Hypothesis (F4) is satisfied.

Finally, the system \eqref{M1} is approximately controllable in $[0, T]$.
	\section{Acknowledgement}
	The first author also would like to thank Council of Scientific and Industrial Research, New Delhi, Government of India (File No. 09/143(0954)/2019-EMR-I), for financial support to carry out his research work. The second author wishes to express his special thanks to SERB DST, New Delhi, for providing a research project with project grant no- CRG/2021/003343/MS.
	\bibliographystyle{elsarticle-num}
	\bibliography{Hilfer}
	
\end{document}